\documentclass[11pt,twoside,a4paper]{amsart}
\usepackage[utf8]{inputenc}
\usepackage[new]{old-arrows}
\usepackage{amsmath,amssymb}
\usepackage{mathrsfs}
\usepackage{amsfonts}
\usepackage{mathtools}
\usepackage{microtype}
\usepackage{amsthm}
\usepackage{tikz-cd}
\usepackage[english]{babel}
\usepackage{xcolor}
\usepackage{graphicx}
\usepackage{url}
\usepackage[hidelinks]{hyperref}
\newcommand{\p}{\mathbb{P}}
\newcommand{\R}{\mathbb{R}}
\newcommand{\Q}{\mathbb{Q}}
\newcommand{\C}{\mathbb{C}}
\newcommand{\s}{\sigma}

\newcommand{\T}{\Tilde{T}}
\newcommand{\X}{\mathscr{X}}

\theoremstyle{plain} 
\newtheorem{thm}{Theorem}[section] 

\newtheorem{cor}[thm]{Corollary} 
\newtheorem{lem}[thm]{Lemma} 
\newtheorem{prop}[thm]{Proposition}
\newtheorem{conj}[thm]{Conjecture} 
\theoremstyle{definition} 
\newtheorem{defn}[thm]{Definition}
\newtheorem{exmp}[thm]{Example}
\theoremstyle{remark} 
\newtheorem{rmk}[thm]{Remark}

\begin{document}

\title[The Hodge conjecture for powers of K3 surfaces of Picard number 16]{The Hodge conjecture for powers of K3 surfaces of Picard number 16}  
\author[M. Varesco]{Mauro Varesco}
\address{Mathematisches Institut, Universit\"at Bonn, Endenicher Allee 60, 53115 Bonn, Germany}
\email{varesco@math.uni-bonn.de}

\begin{abstract}
We study the Hodge conjecture for powers of K3 surfaces and show that if the Kuga--Satake correspondence is algebraic for a family of K3 surfaces of generic Picard number $16$, then the Hodge conjecture holds for all powers of any K3 surface in that family.
\end{abstract}
\maketitle

\section*{Introduction}
\subsection{The determinant and the exceptional Hodge classes}
Let $X$ be a K3 surface, and denote by $T(X)$ its transcendental lattice with its induced Hodge structure of weight two. In the study of the Hodge conjecture for powers of $X$ a central role is played by the endomorphism field $E$ of $T(X)$. If $E$ is a CM field, then the Hodge conjecture for all powers of $X$ is known to hold: Buskin \cite{buskin2019every} and, again, Huybrechts \cite{huybrechts2019motives}, using derived categories, prove the Hodge conjecture for the square of these K3 surfaces. Ram{\'o}n--Mar{\'\i} \cite{mari2008hodge} then shows that for K3 surfaces with a CM endomorphism field, the Hodge conjecture for the square implies the Hodge conjecture for all powers of these K3 surfaces. On the other hand, if $E$ is a totally real field, the situation is more difficult. Unlike the CM case, the algebra of Hodge classes in $\bigotimes^\bullet T(X)$ is never generated by degree-two elements: The determinant of the transcendental lattice $\det T(X)\in T(X)^{\otimes \dim T(X)}$ gives an example of an \textit{exceptional} Hodge class, i.e., a Hodge class in the tensor algebra of $T(X)$ that cannot be expressed in terms of Hodge classes in $T(X)^{\otimes 2}$. In particular, we see that it is not sufficient to prove the Hodge conjecture for the square of these surfaces. Moreover, if the endomorphism field of $X$ is not $\Q$, the determinant is not the unique exceptional Hodge class. Indeed, we prove the following:
\begin{thm}[Theorem \ref{generators in the case of totally real K3}]
Let $X$ be a K3 surface with  totally real endomorphism field $E$. Then, any Hodge class in $\bigotimes^{\bullet}T(X)$ can be expressed in terms of Hodge classes of degree two and the exceptional Hodge classes in $T(X)^{\otimes r},$ where $r\coloneqq \dim_E T(X)$.
\end{thm}
In particular, one sees that $\det T(X)$ decomposes as tensor product of exceptional Hodge classes of lower degree. To see what these exceptional classes are, consider the decomposition $T(X)_\C=\bigoplus_\s V_\s$ of the complex vector space $T(X)_\C$ into eigenspaces  for the action of the field $E$, where the direct sum runs over all embeddings $\sigma\colon E \longhookrightarrow \C$. The classes $\det V_\sigma\in (T(X)_\C)^{\otimes r}$ are of type $(r,r)$, but are not rational. However, it can be shown that the space
$\langle \det V_\s\rangle_{\s}$ is indeed rational, i.e., spanned by Hodge classes. These are the exceptional Hodge classes appearing in the theorem. In particular, to prove the Hodge conjecture for all powers of a K3 surface $X$ with a totally real endomorphism field, one needs to show the Hodge conjecture for the square $X^2$, i.e., prove that all classes in $\mathrm{End}_{\mathrm{Hdg}(X)}(T(X))$ are algebraic, and then show that one of the exceptional classes in $T(X)^{\otimes r}$ is algebraic on $X^r$. Indeed, if one of these exceptional classes is algebraic, then, applying the endomorphisms in $\mathrm{End}_{\mathrm{Hdg}(X)}(T(X))$, we see that all of them are algebraic.

\smallskip

The second implication of our theorem is the fact that the Hodge conjecture does not specialize in families: In a family of K3 surfaces the dimension of the endomorphism field may jump and the dimension of the transcendental lattice may go down, therefore, the Hodge conjecture for powers of the general member does not imply the Hodge conjecture for powers of all other fibres. However, we prove in Proposition \ref{Prop.: determinant is algebraic is a closed property in families} that the algebraicity of $\det T(X)$ specializes also when the dimension of $T(X)$ drops. A similar statement can be proven for the exceptional Hodge classes. This allows us to conclude that the Hodge conjecture specializes in families if the endomorphism field of the special fibre is equal to the endomorphism field of the general
fibre.

\subsection{The main theorem}
A possible approach to tackle the Hodge conjecture for powers of K3 surfaces is via the Kuga--Satake construction, which gives a correspondence between the K3 surface and an abelian variety. When this correspondence is known to be algebraic, it is possible to produce algebraic cycles on powers of the K3 surface from algebraic cycles on powers of the abelian variety. For example, this shows the Hodge conjecture for powers of K3 surfaces that are Kummer surfaces: By Morrison \cite{morrison1984k3}, the Kuga--Satake variety of a Kummer surface is a power of the starting abelian surface, and by Ribet \cite{ribet1983hodge}, the Hodge conjecture holds for powers of abelian surfaces. This same approach has been employed by Schlickewei \cite{schlickewei2010hodge} to prove the Hodge conjecture for the square of K3 surfaces which are double covers of the projective plane branched along six lines using a result by Paranjape \cite{paranjape}.
Inspired by these techniques, we analyze in depth the Kuga--Satake correspondence for K3 surfaces of Picard number $16$, and we prove the following:
\begin{thm}[Theorem \ref{main thm}]
\label{main thm intro}
Let $\mathscr{X}\longrightarrow S$ be a four-dimensional family of K3 surfaces whose general fibre is of Picard number $16$ with an isometry
\[
T(\mathscr{X}_s)\simeq U_\Q^2\oplus \langle a\rangle\oplus\langle b\rangle,
\]
for some negative integers $a$ and $b$. If the Kuga--Satake correspondence is algebraic for the fibres of this family, then the Hodge conjecture holds for all powers of every K3 surface in this family.
\end{thm}
At the time of us writing this article, there are two families of K3 surfaces which satisfy the hypotheses of Theorem \ref{main thm intro}: The family of double covers of $\p^2$ branched along six lines studied in \cite{schlickewei2010hodge} and the family of K3 surfaces which are desingularization of singular K3 surfaces in $\p^4$ with $15$ simple nodes studied in \cite{ingalls2022explicit}. In particular, Theorem \ref{main thm intro} proves the Hodge conjecture for all powers of the K3 surfaces in these families. 
\subsection{Outline of the proof}
The first step of the proof is a study of the Hodge conjecture for powers of the Kuga--Satake varieties of the K3 surfaces of Theorem \ref{main thm intro}. By Lombardo \cite{lombardo01kugasatake}, these abelian varieties are powers of abelian fourfolds of Weil type with discriminant one. We correct in Section
\ref{Sec. Ab. of Weil type} the proof of the following theorem. \begin{thm}[Theorem \ref{Thm.: H.c. for powers of general ab var of weil type}]\cite{abdulali1999abelianIII}
\label{thm. abdulali intro}
Let $A$ be a general abelian variety of Weil type. Then, the Hodge conjecture for $A$ implies the Hodge conjecture for all powers $A^k$.
\end{thm}
In a recent preprint, Milne \cite{milne2021tate} gives an alternative proof of Theorem \ref{thm. abdulali intro} using an equality between the Mumford--Tate group and the algebraic group which preserves the algebraic classes on the powers of the abelian variety\footnote{Thanks to Bert van Geemen for the reference.}. Our proof has the advantage of giving a concrete description of all Hodge classes on the powers of a general abelian variety of Weil type. For example this allows us to show that if the Kuga--Satake correspondence for the general K3 surface of Picard number $16$ with transcendental lattice as in Theorem \ref{main thm intro} is algebraic, then the Hodge conjecture holds for the resulting abelian fourfold of Weil type with discriminant one. Indeed we prove the following:
\begin{prop}[Prop. \ref{prop. KSH implies H conj for ab 4}]
\label{Prop: hodge conj for the resulting ab 4fold intro}
Let $X$ be a general K3 surface of Picard number $16$ as in Theorem \ref{main thm intro}, and denote by $A$ the abelian fourfold of Weil type with discriminant one appearing as simple factor of the Kuga--Satake variety of $X$. Then, if the Kuga--Satake correspondence is algebraic for $X$, the Weil classes on $A$ are algebraic. Thus, the Hodge conjecture holds for $A$ and, hence, for all powers $A^k$.
\end{prop}
The Hodge conjecture for general abelian fourfolds of Weil type  with discriminant one
has already been proven by Markman \cite{markman2022monodromy} using generalized Kummer varieties. Our approach is different but
no new case of the Hodge conjecture is established. Proposition \ref{Prop: hodge conj for the resulting ab 4fold intro} is included to highlight the strong link between 
algebraicity of the Kuga–-Satake correspondence and the Hodge conjecture for the Kuga--Satake variety.
Assuming the algebraicity of Kuga--Satake correspondence, the Hodge conjecture for powers of the Kuga--Satake variety implies the Hodge conjecture for the K3 surface. This allows us deduce the Hodge conjecture the general K3 surface of the family of Theorem \ref{main thm intro}. Similarly, we prove the Hodge conjecture for all powers of K3 surfaces of Picard number $16$ and a totally real endomorphism field of degree two. To establish the same for the K3 surfaces of higher Picard number with totally real endomorphism field, we rely on the aforementioned result stating that the Hodge conjecture specializes in families if the endomorphism field of the special fibre is equal to the endomorphism field of the general fibre. This last hypothesis is satisfied, as the endomorphism field of K3 surfaces of Picard number higher than $16$ is $\Q$ if it is not a CM field.

\section*{Acknowledgments}
I would like to thank my PhD.\ supervisor Daniel Huybrechts for suggesting me this topic and for his continuous advice and support. I am also grateful to Ben Moonen and Bert van Geemen for reviewing a preliminary version of this article. Our discussion of the role of $\det T(X)$ was influenced by Claire Voisin and Qizheng Yin. This research was funded by ERC Synergy Grant HyperK, Grant
agreement ID 854361.
\section{The Hodge conjecture and Kuga--Satake varieties}
\label{Section Hodge conjecture and Kuga-Satake}
In this section, we recall the Hodge conjecture in the special case of powers of K3 surfaces and review the construction of the Kuga--Satake varieties. For a complete introduction we refer to \cite{huybrechtsK3surfaces}.

\subsection{The Hodge conjecture}
Let $X$ be a smooth complex projective variety. For a non-negative integer $k$ denote by $H^{k,k}(X,\Q)\coloneqq H^{k,k}(X,\C)\cap H^{2k}(X,\Q)$ the set of Hodge classes of degree $k$.
\begin{conj}[Hodge conjecture for powers of K3 surfaces] Let $X$ be a K3 surface, and let $k$ and $n$ be positive integers. Then, the $\Q$-algebra of Hodge classes in $H^{k,k}(X^n,\Q)$ is generated by cohomology classes of algebraic cycles on $X^n$.
\end{conj}
As the K3 surface $X$ is projective by assumption, the intersection pairing on $X$ induces the direct sum decomposition
\[
H^2(X,\Q)=\mathrm{NS}(X)_\Q\oplus T(X),
\]
where $\mathrm{NS}(X)$ is the Néron--Severi group of $X$, and $T(X)$ is its transcendental lattice, i.e., the smallest rational sub-Hodge structure $T$ of $H^2(X,\Q)$ such that $T^{2,0}=H^{2,0}(X)$.
\begin{lem}
\label{Lem.: Hodge conjecture and Transcendental lattice}
The Hodge conjecture holds for all powers of a K3 surface $X$ if and only if all Hodge classes in the tensor algebra of $T(X)$ are algebraic.
\begin{proof}
Let $n$ and $k$ be positive integers. By K\"{u}nneth decomposition, we have that
\[
H^{2k}(X^n,\Q)\cong\bigoplus \left( T(X)^{\otimes a}\otimes \mathrm{NS}(X)_\Q^{\otimes b}\otimes H^0(X,\Q)^{\otimes c}\otimes H^4(X,\Q)^{\otimes d}\right),
\]
where the direct sum runs over all non-negative integers $a,b,c,$ and $d$ satisfying $2k=2a+2b+4d$ and $a+b+c+d=n$.
Since all elements in $\mathrm{NS}(X)_\Q,$ $H^0(X,\Q)$, and $H^4(X,\Q)$ are obviously algebraic, to prove the Hodge conjecture for the powers of $X$, it suffices to show that the Hodge classes in the algebra $\bigotimes^{\bullet}T(X)$ are algebraic.
\end{proof}
\end{lem}
Let us recall the notion of Hodge group of $T(X)$, which we use to study the algebra of Hodge classes in $\bigotimes^{\bullet}T(X)$.
Let $V$ be a rational Hodge structure, and denote by $\rho\colon \C^*\longrightarrow \mathrm{GL}(V_\R)$ the morphism defining the Hodge structure on $V$. The \textit{Hodge group} of $V$ is defined as the smallest algebraic sub-group $\mathrm{Hdg}(V)$ of $\mathrm{GL}(V)$ defined over $\Q$ such that
\[
\rho(S^1)\subseteq \mathrm{Hdg}(V)(\R),
\]
where $S^1\subseteq \C^*$ is the unit circle.
Hodge classes can be equivalently defined as the classes that are invariant under he action of the Hodge group, indeed, the following holds: 
\begin{lem}\cite[Sec.\ 2]{ribet1983hodge}
\label{Lem.: Hodge classes= invariant classes}
Let $a$ and $b$ be non-negative integers. An element $v\in V^{\otimes a}\otimes (V^*)^{\otimes b}$ is a Hodge class if and only if it is invariant under the action of $\mathrm{Hdg}(V)$.\qed
\end{lem}
Given a K3 surface $X$, the algebra of Hodge classes in $\bigotimes^{\bullet} T(X)$ is then the invariant algebra $(\bigotimes^{\bullet} T(X))^{\mathrm{Hdg}(X)}$, where $\mathrm{Hdg}(X)$ is the Hodge group of $T(X)$. Similarly, given an abelian variety $A$, denote by $\mathrm{Hdg}(A)$ the Hodge group of $H^1(A,\Q)$. Considering the natural embedding of Hodge structures 
$\bigwedge^{\bullet}H^1(A,\Q)\longhookrightarrow \bigotimes^{\bullet}H^1(A,\Q)$, we see that $(\bigwedge^{\bullet}H^1(A,\Q))^{\mathrm{Hdg}(A)}$ is
the algebra of Hodge classes in $\bigwedge^{\bullet}H^1(A,\Q)$.  

\subsection{Kuga--Satake varieties}
\label{Sec. Kuga-Satake}
We shortly recall the Kuga--Satake construction following \cite{van2000kuga} and \cite[Ch.\ 4]{huybrechtsK3surfaces}.

\smallskip

Let $(V,q)$ be a polarized rational Hodge structure of weight two of K3-type, i.e., $\dim V^{2,0}=1$. The \textit{Clifford algebra} of $V$ is the quotient of the tensor algebra of $V$ by the two-sided ideal generated by elements of the form $v\otimes v-q(v)$ for $v\in V$:
\[\mathrm{Cl}(V)\coloneqq \textstyle\bigotimes^{\bullet}V/\langle v\otimes v-q(v)\rangle.\]
Denote by $\mathrm{Cl}^+(V)$ the subalgebra of $\mathrm{Cl}(V)$ generated by the elements of even degree. As shown in \cite[Prop.\ 2.6]{huybrechtsK3surfaces}, the Hodge structure on $V$ induces a Hodge structure of weight one on $\mathrm{Cl}^+(V)$ for which there exists an embedding of Hodge structures $V\longhookrightarrow \mathrm{Cl}^+(V)\otimes \mathrm{Cl}^+(V).$
\begin{defn}
The \textit{Kuga--Satake variety} of $(V,q)$ is an abelian variety $\mathrm{KS}(V)$ such that there is an isomorphism of Hodge structures
$H^1(\mathrm{KS}(V),\Q)\simeq\mathrm{Cl}^+(V)$.
\end{defn}
Note that the abelian variety $\mathrm{KS}(V)$ is determined only up to isogeny. For our purposes, this description of the Kuga--Satake variety up to isogeny is sufficient.

\smallskip
Let $X$ be a K3 surface, and denote by $\mathrm{KS}(X)$ the Kuga--Satake variety of $T(X)$. By construction, there is an embedding of Hodge structures:
\[
T(X)\longhookrightarrow H^1(\mathrm{KS}(X),\Q)\otimes H^1(\mathrm{KS}(X),\Q)\subseteq H^2(\mathrm{KS}(X)^2,\Q).
\]
Composing this map with the natural projection $H^2(X,\Q)\longrightarrow T(X)$ induced by the polarization on $X$, we obtain the following morphism of Hodge structures:
\[H^2(X,\Q) \longrightarrow  H^2(\mathrm{KS}(X)^2,\Q).\]
This morphism is called \textit{Kuga--Satake correspondence}. By Poincaré duality, this map is induced by a Hodge class $\kappa\in H^{2,2}(X\times \mathrm{KS}(X)^2,\Q)$.
The Hodge conjecture then predicts the following:
\begin{conj}[Kuga--Satake Hodge conjecture for K3 surfaces]
Let $X$ be a K3 surface. Then, the Hodge class $\kappa$ is algebraic. 
\end{conj}
If the Kuga--Satake correspondence is algebraic, it is possible reduce the study of the Hodge conjecture on powers of K3 surfaces to the study of the Hodge conjecture for powers of abelian varieties. Indeed, as in the proof of \cite[Thm. 2]{schlickewei2010hodge}, the following holds:
\begin{lem}
\label{Lem.: Kuga-Satake corresponedence and Hodge conjecture}
Let $X$ be a K3 surface for which the Kuga--Satake correspondence is algebraic. Then, a Hodge class in the tensor algebra of $T(X)$ is algebraic if and only if its image via the Kuga--Satake correspondence is algebraic.
\begin{proof}
Let $\alpha$ be a Hodge class in $T(X)^{\otimes k}$ for some $k$. If $\alpha$ is algebraic on $X^k$ then also its image via the Kuga--Satake correspondence is algebraic as we are assuming that the Kuga--Satake Hodge conjecture for $X$. Conversely, applying \cite[Cor.\ 3.14]{kleiman1968algebraic} to the Kuga--Satake correspondence, we see that there is an algebraic projection 
\[
H^2(\mathrm{KS}(X)^2,\Q)\longrightarrow T(X)\subseteq H^2(X,\Q).
\]
Therefore, if the image of $\alpha$ is algebraic on $\mathrm{KS}(X)^{2k}$ also $\alpha$ is algebraic on $X^k$.
\end{proof}
\end{lem}

\section{Generators of the algebra of Hodge classes}
\label{Section: On generators of the algebra of Hodge classes}
In this section, using the techniques introduced by Ribet \cite{ribet1983hodge}, we study the algebra of Hodge classes for the powers of K3 surfaces. A similar study has been done in \cite{mari2008hodge}. However, there is a mistake in the proof of \cite[Prop.\ 5.2]{mari2008hodge} that leads to a wrong conclusion in the case of K3 surfaces with totally real multiplication.

\smallskip

By Lemma \ref{Lem.: Hodge conjecture and Transcendental lattice} and Lemma \ref{Lem.: Hodge classes= invariant classes}, to study the Hodge conjecture for powers of $X$ it suffices to investigate it for the algebra $\left(\bigotimes^{\bullet} T(X)\right)^{\mathrm{Hdg}(X)}$. To ease the exposition, let us introduce some terminology we will use in the following sections. 
An element $f \in \bigotimes^{\bullet} T(X)$ is called \textit{homogeneous of degree} $d$ if $f\in T(X)^{\otimes d}$. Note that the group of permutations $\mathfrak{S}^d$ naturally acts on $T(X)^{\otimes d}$.
\begin{defn}
\label{def.: meaning of expressed in terms of}
We say that homogeneous elements $e_1,\ldots, e_r\in\left(\bigotimes^{\bullet} T(X)\right)^{\mathrm{Hdg}(X)}$ generate $\left(\bigotimes^{\bullet} T(X)\right)^{\mathrm{Hdg}(X)}$ if any element $f\in \left(\bigotimes^{\bullet} T(X)\right)^{\mathrm{Hdg}(X)}$ can be written as a sum $f=\sum_i f_i$, where each $f_i$ is homogeneous and, up to permutation and up to a scalar, it is a tensor product of elements in $\{e_1,\ldots,e_k\}$. In this case, we say that any element in $\left(\bigotimes^{\bullet} T(X)\right)^{\mathrm{Hdg}(X)}$ can be expressed in terms of $e_1,\ldots,e_k$.
\end{defn}
Note that if $\{e_1,\ldots, e_r\}$ is a set of generators of $\left(\bigotimes^{\bullet} T(X)\right)^{\mathrm{Hdg}(X)}$, to prove the Hodge conjecture for the powers of $X$, it suffices to show that the classes $e_i$ are algebraic.

\smallskip
Let $E\coloneqq \mathrm{End}_{\mathrm{Hdg}(X)}(T(X))$ be the endomorphism algebra of $T(X)$. As $T(X)$ is an irreducible Hodge structure, $E$ is a field, cf. \cite[Sec.\ 3.2]{huybrechtsK3surfaces}. We call it the \textit{endomorphism field} of $X$. Let $\psi$ be the polarization on $T(X)$ induced by the intersection form on $H^2(X,\mathbb{Q})$. Note that $\psi$ is symmetric, since the Hodge structure on $T(X)$ has weight two. The \textit{Rosati involution} is the involution on $E$ sending an element $e$ to the element $e'$ for which 
\[\psi(e(x),y)=\psi(x,e'(y)), \quad \forall x,y\in T(X).\]
As one checks, $F\coloneqq \{e\in E\,|\, e'=e\}\subseteq E$ is a totally real field, and either $E=F$ or $E$ is a CM field with maximal totally real sub-field $F$, see \cite[Thm.\ 3.3.7]{huybrechtsK3surfaces}.
We treat these two cases separately.

\subsection{K3 surfaces with a CM endomorphism field}
\label{Section: CM case}
In this section, we prove the following:
\begin{thm}\cite[Prop.\ 5.2]{mari2008hodge}
\label{Thm: invariant degree 2 elements generate in the CM case}
Let $X$ be a K3 surface whose endomorphism field is a CM field. Then, any Hodge class $\bigotimes^{\bullet}T(X)$ can be expressed in terms of Hodge classes in $T(X)^{\otimes 2}$.
\end{thm}
Note that in the reference it is stated that the same holds also in the case of K3 surfaces with totally real multiplication. In the next section, we show that this is not true.

\smallskip

Before proving Theorem \ref{Thm: invariant degree 2 elements generate in the CM case}, let us deduce from it the Hodge conjecture for all powers of the K3 surface $X$:
\begin{cor}\cite[Thm.\ 5.4]{mari2008hodge}
\label{Cor.: H.c. for K3 of CM type}
Let $X$ be a complex, projective K3 surface whose endomorphism field is a CM field. Then, the Hodge conjecture holds for all powers of $X$.
\begin{proof}
By Lemma \ref{Lem.: Hodge conjecture and Transcendental lattice}, in order to prove the Hodge conjecture for all powers of $X$, we need to show that the Hodge classes in $\bigotimes^{\bullet} T(X)$ are algebraic. By Theorem \ref{Thm: invariant degree 2 elements generate in the CM case}, to show this, it suffices to show that every Hodge class in $T(X)^{\otimes 2}$ is algebraic. This has been proven in \cite{buskin2019every} and again in \cite[Cor.\ 0.4.ii]{huybrechts2019motives}, where the authors prove the Hodge conjecture for the square of a K3 surface with CM endomorphism field.
\end{proof}
\end{cor}

The remainder of this section is dedicated to the proof of Theorem \ref{Thm: invariant degree 2 elements generate in the CM case}.
Let us start by recalling the following result from linear algebra:
\begin{lem}\cite[Lem.\ 4.3]{deligne1982hodge}
\label{Lem: Trace}
Let $k$ be a field and let $V$ be a vector space of finite dimension over a finite separable field extensions $k'$ of $k$. Then, the map
\[\mathrm{Hom}_{k'}(V,k')\longrightarrow \mathrm{Hom}_k(V,k),\quad f\longmapsto\mathrm{Tr}_{k'/k}\circ f,\]
is an isomorphism of $k$-vector spaces.
\end{lem}
\begin{lem}\cite[Sec.\ 2.1]{zarhin1983hodge}
\label{Lem: varphi CM K3 surface}
Let $X$ be a K3 surface with a CM endomorphism field $E$ and let $\psi$ be the polarization on $T(X)$ induced by the intersection pairing. Then, there exists a unique non-degenerate $E$-Hermitian map $\varphi\colon T(X)\times T(X)\longrightarrow  E$ which satisfies $\psi(v,w)=\mathrm{Tr}_{E/\Q}(\varphi(v,w))$ for every $v,w\in T(X)$.
\begin{proof}
Denote by $T(X)^\dag$ the space $T(X)$ with $E$ acting on it via complex conjugation. The polarization $\psi$ can be seen as a $\Q$-linear morphism $T(X)\otimes_E T(X)^\dag\longrightarrow \Q$.
Lemma \ref{Lem: Trace} then says that there exists a unique $E$-linear map $\varphi\colon T(X)\otimes_E T(X)^\dag\longrightarrow E$ such that $\psi(v,w)=\mathrm{Tr}_{E/\Q}(\varphi(v,w))$ for every $v,w\in T(X)$.
Viewing $\varphi$ as a map $T(X)\times T(X)\longrightarrow E$, we see that it is $E$-Hermitian and satisfies all the required proprieties.
\end{proof}
\end{lem}

Let $\mathrm{Hdg}(X)$ be the Hodge group of the transcendental lattice of $X$. As the polarization $\psi$ is  a morphism of Hodge structures 
\[
\psi\colon T(X)\otimes T(X)\longrightarrow \Q(-2),
\]
the action of $\mathrm{Hdg}(X)$ on $T(X)$ preserves $\psi$, i.e., $\psi(Av,Aw)=\psi(v,w)$ for every $v,w\in T(X)$ and every $A\in \mathrm{Hdg}(X)$. By construction of $\varphi$, we conclude that $\mathrm{Hdg}(X)\subseteq \mathrm{Res}_{F/\Q}(\mathrm{U}(T(X),\varphi))$, where $\mathrm{U}(T(X),\varphi)$ is the unitary group with respect to the $E$-Hermitian form $\varphi$, and $F$ is the maximal totally real sub-field of $E$. In \cite[Thm.\ 2.3.1]{zarhin1983hodge}, it is shown that this inclusion is always an equality, i.e., that the Hodge group of $X$ satisfies
$\mathrm{Hdg}(X)=\mathrm{Res}_{F/\Q}(\mathrm{U}(T(X),\varphi)).$
By Lemma \ref{Lem.: Hodge classes= invariant classes}, the ring of Hodge classes in $\bigotimes^{\bullet}T(X)$ is then equal to
\begin{equation*}
    \label{eq.: hdg classes in T(X) in the CM case}
\textstyle(\bigotimes^{\bullet}T(X))^{\mathrm{Hdg}(X)}=(\bigotimes^{\bullet}T(X))^{\mathrm{Res}_{F/\Q}(\mathrm{U}(T(X),\varphi))}.
\end{equation*}
Let $\mathrm{Hdg}(X)(\mathbb{C})$ be the group of $\C$-valued points of $\mathrm{Hdg}(X)$. As in \cite[Sec.\ 6.7]{van1994introduction}, there is an isomorphism of graded algebras 
\begin{equation}
\label{eq. invariants and scalar extension}
    \textstyle(\bigotimes^{\bullet} T(X))^{\mathrm{Hdg}(X)}\otimes_{\Q}\C\simeq(\bigotimes^{\bullet}T(X)_\C)^{\mathrm{Hdg}(X)(\mathbb{C})}.
\end{equation}
From this, using the terminology of Definition \ref{def.: meaning of expressed in terms of}, we deduce the following:
\begin{lem}
\label{Lem: extension of scalars}
If $\widetilde{e}_1,\ldots, \widetilde{e}_r$ are homogeneous generators of $(\bigotimes^{\bullet}T(X)_\C)^{\mathrm{Hdg}(X)(\C)}$ over $\C$, then there exist generators $e_1,\ldots, e_r$ of $(\bigotimes^{\bullet}T(X))^{\mathrm{Hdg}(X)}$ over $\Q$, such that $e_i$ is homogeneous with $\deg e_i=\deg \widetilde{e}_i$ for every $i$.
\end{lem}
In particular, to show that any element in $(\bigotimes^{\bullet}T(X))^{\mathrm{Hdg}(X)}$ can be expressed in terms of degree-two elements, it suffices to show that any element in $(\bigotimes^{\bullet}T(X)_\C)^{\mathrm{Hdg}(X)(\C)}$ can be expressed in terms of degree-two elements.
Let $F$ be the maximal totally real subfield of $E$. Extending scalars to $\R$, we get the following well known fact:
\begin{lem}
\label{lem: Hermitian form on V_s}
The real vector space $T(X)_\R$ decomposes as a direct sum 
\[T(X)_\R=\displaystyle\bigoplus_{\sigma\colon F\hookrightarrow \R}V_\s,\]
where $V_\s$ are real vector spaces with an $E$-action. Moreover, this decomposition is $\varphi$-orthogonal and $\varphi$ induces a non-degenerate $\C$-Hermitian form  on $V_\s$ for every $\s$.
\begin{proof}
As $T(X)$ is a free $F$-module, $T(X)\otimes_\Q\R$ is a free $F\otimes_\Q\R$-module. Using the isomorphism 
\[F\otimes_\Q\R\simeq \prod_{\sigma\colon F\longhookrightarrow\R}\R, \quad e\otimes r\longmapsto (\s(e)r)_\s,\]
we see that 
\[T(X)_\R=\displaystyle\bigoplus_{\sigma\colon F\hookrightarrow \R}V_\s,\]
where $V_\s\coloneqq T(X)\otimes_{F,\s}\R\simeq\{v\in T(X)\otimes_{\Q}\R \,| \,f(v)=\s(f)v \quad \forall f\in F\}.$
Let us now show that $E$ acts on $V_\s$. Let $\sigma\colon F\longhookrightarrow \mathbb{R}$ be an embedding, and let $v\in V_\s$ be any element. For every $e\in E$ and every $f\in F$, we have that
\[f(e(v))=e(f(v))=e(\sigma(f)v)=\sigma(f)e(v),\]
where the last equality follows from the fact that the action of $E$ on $T(X)_\R$ is the $\R$-linear extension of the action of $E$ on $T(X)$. This shows that $e(v)\in V_\s$ and that the action of $E$ on $T(X)_\R$ induces an action on $V_\s$.
To see that this decomposition is $\varphi$-orthogonal, let $\s,\widetilde{\s}\colon F\longhookrightarrow \R$ be two different embeddings and choose $f\in F$ such that $\s(f)\not=\widetilde{\s}(f)$. Since $\varphi$ is $E$-Hermitian and since the elements of $F$ are fixed by the Rosati involution, we have that
\[\s(f)\varphi(v,w)=\varphi(f(v),w)=\varphi(v,f(w))=\widetilde{\s}(f)\varphi(v,w),\]
for every $v\in V_\s$ and $w\in V_{\Tilde{\s}}$. From this, we deduce that $\varphi(v,w)=0$, i.e., $\varphi|_{V_\s\times V_{\widetilde{\s}}}=0$.
To define the $\C$-Hermitian form $\varphi_\s$ on $V_\s$, extend the embedding $\s\colon F\longhookrightarrow\R$ to an embedding $\tau\colon E\longhookrightarrow\C$, which can be done since $E$ is a totally imaginary quadratic extension of $F$.
Let $\varphi_\s$ be the composition
\[\begin{tikzcd}
\varphi_\s\colon V_\s\times V_\s \arrow[r, "{\varphi\otimes_{F,\s}\R}"] & {E\otimes_{F,\s}\R} \arrow[r, hook, "{\tau\otimes_{F,\s}\R}"] & \C.
\end{tikzcd}
\]
 From the fact that the decomposition $T(X)_\R=\bigoplus_\s V_\s$ is $\varphi$-orthogonal, it follows that $\varphi_\s$ is a non-degenerate $\C$-Hermitian form on $V_\s$. This concludes the proof.
\end{proof}
\end{lem}

\begin{rmk}
\label{rmk.: unitary group}
Since the action of $\mathrm{Hdg}(X)$ on $T(X)$ is $E$-linear, it preserves the eigendecomposition $T(X)_\R=\bigoplus_\s V_\s$ of Lemma \ref{lem: Hermitian form on V_s}. From this, one deduces that the group of real valued points of $\mathrm{Hdg}(X)$ decomposes as a product 
\[\mathrm{Hdg}(X)(\R)=\prod_{\s\colon F\hookrightarrow\R} U_\s(\R),\]
where $U_\s\coloneqq \mathrm{U}(V_\s,\varphi_\s)$.
Therefore, the invariants in $(\bigotimes^{\bullet}T(X)_\R)^{\mathrm{Hdg}(X)(\R)}$ can be expressed in terms of invariants in the spaces $(\bigotimes^{\bullet}V_\s)^{U_\s(\R)}$. 
\end{rmk}
By definition of CM field, $E=F(\rho)$ with $\rho^2\in F$ such that $\tilde{\s}(\rho^2)$ is negative for any embedding $\tilde{\s}\colon F\longhookrightarrow \R$. Fix $\s\colon F\longhookrightarrow \R$ an embedding, and let $\lambda\in \R$ such that $\s(\rho^2)=-\lambda^2$. As we have seen in Lemma \ref{lem: Hermitian form on V_s}, the field $E$ acts on $V_\s$. The action of $\frac{\rho}{\lambda}$ then induces a 
$\C$-vector space structure on $V_\s$. In particular, $V_\s$  is even-dimensional and there is a decomposition
$(V_\s)_\C=V^{1,0}\oplus V^{0,1}$ with $V^{1,0}\coloneqq \{v\in (V_\s)_\C\, |\, \rho (v)=i\lambda v\}$ and $V^{0,1}\coloneqq \{v\in (V_\s)_\C\, |\, \rho (v)=-i\lambda v\}$.
Let $\omega_\s$ be the $\C$-linear extension of the map $\mathrm{Im}\varphi_\s\colon V_\s\times V_\s\longrightarrow \R$. Since $\varphi_\s$ is non-degenerate and $E$-Hermitian, $\omega_\s$ is a non-degenerate symplectic form on $(V_\s)_\C$. 
\begin{lem}
\label{Lem: (V_s)_C=Y+Z}
The decomposition $(V_\s)_\C=V^{1,0}\oplus V^{0,1}$ of the $\C$-vector space $(V_\s)_\C$ is Lagrangian with respect to the symplectic form $\omega_\s$, i.e., 
\[\omega_\s|_{V^{1,0}\times V^{1,0}}=0\quad \text{and}\quad \omega_\s|_{V^{0, 1}\times V^{0,1}}=0.\]
Furthermore, $\omega_\s$ induces an isomorphism of complex vector spaces $V^{0,1}\simeq (V^{1,0})^*$.
\begin{proof}
The space $V^{1,0}$ is Lagrangian since for $v,v'\in V^{1,0}$ the following holds:
\[\omega_\s(v,v') =\frac{1}{\lambda^2 }\omega_\s(\rho (v),\rho (v'))=\frac{1}{\lambda^2 }\omega_\s(i\lambda v,i\lambda v')\\
=-\omega_\s(v,v'),\]
where the first equality follows from the fact that $\varphi$ is $E$-Hermitian. A similar argument proves that $V^{0,1}$ is Lagrangian. The second assertion then follows from the fact that $\omega_\s$ is non-degenerate.
\end{proof}
\end{lem}
\begin{rmk}
\label{Rmk: subgroups of GL(V_s)}
As Hermitian forms are determined by their symplectic imaginary part, we deduce that $U_\s(\C)\simeq \mathrm{Sp}((V_\s)_\C,\C)$. Identifying these two groups, we can consider $U_\s(\C)$ as a subgroup of $\mathrm{GL}((V_\s)_\C,\C)$.
On the other hand, we can consider $\mathrm{GL}(V^{1,0},\C)$ as a subgroup of $\mathrm{GL}((V_\s)_\C,\C)$ via the embedding
\[
A\longmapsto
\begin{pmatrix}
A & 0\\
0 & A^*
\end{pmatrix}\in \mathrm{GL}(V^{1,0}\oplus(V^{1,0})^*,\C)\simeq\mathrm{GL}((V_\s)_\C,\C),
\]
where $A^*$ denotes the natural action of $A$ on $(V^{1,0})^*$.
\end{rmk}

\begin{lem}
\label{Lem: the action of S}
The groups $U_\s(\C)$ and $\mathrm{GL}(V^{1,0},\C)$ coincide when viewed as subgroups of $\mathrm{GL}((V_\s)_\C,\C)$ as in Remark \ref{Rmk: subgroups of GL(V_s)}.
\begin{proof}
Let $A\in \mathrm{GL}(V^{1,0},\C)$. By construction, $A$ sends an element $v\in V^{0,1}$ to the element of $V^{0,1}$ corresponding to the element $\omega_\s(A^{-1}(\cdot),v)\in (V^{1,0})^*$ via the isomorphism of Lemma \ref{Lem: (V_s)_C=Y+Z}. This implies that the action of $A$ on $(V_\s)_\C$ preserves the symplectic form $\omega_\s$, i.e., $A$ belongs to $U_\s(\C)$. To prove the other inclusion, note that the action of $U_\s(\C)$ is $E$-linear and hence preserves the direct sum decomposition $(V_\s)_\C=V^{1,0}\oplus V^{0,1}$.
Therefore, if $A$ is an element of $U_\s(\C)$, its restriction $A|_{V^{1,0}}$ is an invertible matrix on $V^{1,0}$. Then, via the isomorphism $\mathrm{GL}(V_\s,\C)\simeq \mathrm{GL}(V^{1,0}\oplus(V^{1,0})^*,\C)$, the matrix $A$ is sent to
$\begin{pmatrix}
A|_{V^{1,0}} & 0\\
0 & (A|_{V^{1,0}})^*
\end{pmatrix}.
$
This shows that $A\in\mathrm{GL}(V^{1,0},\C)$. As we have shown the double inclusion, we conclude that the two groups are the same.
\end{proof}
\end{lem}

We are now able to prove Theorem \ref{Thm: invariant degree 2 elements generate in the CM case}:
\begin{proof}[Proof of Theorem \ref{Thm: invariant degree 2 elements generate in the CM case}]
By the previous discussion and Remark \ref{rmk.: unitary group}, we just need to prove that any element in $(\bigotimes^{\bullet}(V_\s)_\C)^{U_\s(\C)}$ can be expressed in terms of elements of degree two. By Lemma \ref{Lem: (V_s)_C=Y+Z}, this algebra can be identified with $(\bigotimes^{\bullet}(V^{1,0}\oplus
(V^{1,0})^*))^{\mathrm{GL}(V^{1,0},\C)}$.
Given $f\in \left(\bigotimes^p(V^{1,0}\oplus (V^{1,0})^*)\right)^{\mathrm{GL}(V^{1,0},\C)}$ an invariant element of degree $p$, view it as a map
\[f\colon((V^{1,0})^*\oplus V^{1,0})\otimes\cdots\otimes ((V^{1,0})^*\oplus V^{1,0})\longrightarrow \C.\]
Decomposing $f$ as a sum $\sum_i f_i$ where each $f_i$ is a function $W_1\otimes\cdots\otimes W_p\longrightarrow \C$, with $W_j=V^{1,0}$ or $(V^{1,0})^*$ for all $j$, we see that it suffices to show that each $f_i$ can be expressed in terms of invariant elements of degree two.
By construction, up to permuting its factors, \[f_i\in\left(\mathrm{Hom}((V^{1,0})^{\otimes r}\otimes ((V^{1,0})^*)^{\otimes s},\C)\right)^{\mathrm{GL}(V^{1,0},\C)}\] for some $r$ and $s$ with $r+s=p$.
According to the fundamental theorem of invariants for $\mathrm{GL}(V^{1,0},\C)$ as presented in \cite[Thm.\ 4.2]{kraft1996classical}, the ring \[\left(\mathrm{Hom}((V^{1,0})^{\otimes r}\otimes ((V^{1,0})^*)^{\otimes s},\C)\right)^{\mathrm{GL}(V^{1,0},\C)}\] is non-trivial if and only if $r=s$ and in this case it is generated by complete contractions, i.e., maps of the form
\[(V^{1,0})^{\otimes s}
\otimes ((V^{1,0})^*)^{\otimes s}\longrightarrow\C, \quad  v_1\otimes\cdots\otimes v_s\otimes \mu_1\otimes\cdots\otimes\mu_s \longmapsto\textstyle\prod_i \mu_i(v_{\sigma(i)}),\]
for some $\sigma\in \mathfrak{S}_s$.
This statement of invariant theory can be found also in \cite[Ch.\ III]{h.weyl2016classical}. 
This concludes the proof of Theorem \ref{Thm: invariant degree 2 elements generate in the CM case}: Indeed, any complete contraction can be written (up to permuting its factors) as a tensor product of complete contractions of degree two: \[V^{1,0}\otimes (V^{1,0})^*\longrightarrow \C,\quad v\otimes \mu\longmapsto \mu(v),\]
which are invariant.
\end{proof}

\subsection{K3 surfaces with a totally real endomorphism field}
Let $X$ be a complex, projective K3 surface with totally real endomorphism field $E= \mathrm{End}_{\mathrm{Hdg}(X)}(T(X))$. In this case, the Rosati involution is the identity, and, for every $e\in E$ and every $v, w\in T(X)$, we have that
\[\psi(ev,w)=\psi(v,ew).\]
As in Lemma \ref{Lem: varphi CM K3 surface}, one shows that there exists a non-degenerate symmetric $E$-bilinear map $\varphi\colon T(X)\times T(X)\longrightarrow E$ such that $\psi(v,w)=\mathrm{Tr}_{E/\Q}(\varphi(v,w))$. Note that this time $\varphi$ is $E$-bilinear since $E$ is totally real. Moreover, by \cite[Thm.\ 2.2.1]{zarhin1983hodge}, the following equality holds
\[
\mathrm{Hdg}(X)=\mathrm{Res}_{E/\Q}(\mathrm{SO}(T(X),\varphi)).
\]
Similarly to Lemma \ref{lem: Hermitian form on V_s}, we have the following result:
\begin{lem}
\label{lem: Symmetric form on V_s}
The complex vector space $T(X)_\C\coloneqq T(X)\otimes_{\Q}\C$ decomposes as a direct sum 
\[T(X)_\C=\displaystyle\bigoplus_{\sigma\colon E\hookrightarrow \C}V_\s,\]
where $V_\s$ are complex vector spaces with an $E$-action. Moreover, the action of $\mathrm{Hdg}(X)$ preserves this decomposition, and the $E$-bilinear symmetric form $\varphi$ induces a non-degenerate $\C$-bilinear symmetric form $\varphi_\s$ on $V_\s$ for any $\s$. 
\qed
\end{lem}
\begin{rmk}
Note that the decomposition of the transcendental lattice into $E$-eigenspaces of Lemma \ref{lem: Symmetric form on V_s} already holds over $\R$ as $E$ is totally real. We extended scalars directly to $\C$ to introduce as little notation as necessary.
\end{rmk}
Similarly to the CM case, we deduce from the isomorphsim
\[
\textstyle(\bigotimes^{\bullet}T(X))^{\mathrm{Hdg}(X)}\otimes_\Q\C\simeq \textstyle(\bigotimes^{\bullet}T(X)_\C)^{\mathrm{Hdg}(X)(\C)}
\]
that the invariants in $\textstyle(\bigotimes^{\bullet}T(X)_\C)^{\mathrm{Hdg}(X)(\C)}$ can be expressed in terms of invariants in the spaces $\left(\bigotimes^{\bullet}V_\s\right)^{\mathrm{SO(V_\s,\C)}}$. 
The complex vector space $\bigwedge^{\dim V_\s}V_\s\subseteq V_\s^{\otimes\dim V_\s}$ is one-dimensional and thus determines a unique element up to a complex scalar, denote it by $\det V_\s$. Note that $\det V_\s$ is invariant under the action of $\mathrm{SO}(V_\s,\C)$.
With this notation and with the convention of Definition \ref{def.: meaning of expressed in terms of}, the following holds:
\begin{thm}\cite[Thm.\ 10.2]{kraft1996classical}
\label{Thm: invariants of SO(n)}
Let $V_\s$ be a complex vector space endowed with a non-degenerate symmetric $\C$-bilinear form.
Then, the $\mathrm{SO}(V_\s,\C)$-invariants in $\bigotimes^{\bullet}V_\s$ can be expressed in terms of $\mathrm{SO}(V_\s,\C)$-invariants of degree two and $\det V_\s$.
\begin{proof}
To deduce this statement from the one in the reference, just note that, for any positive integer $k$, any element in $V_\sigma^{\otimes k}$ can be viewed as a homogeneous polynomial function $(V_\s^*)^{\oplus k}\longrightarrow \C$ of degree $k$: For example, the map $(V_\s^*)^{\oplus k}\longrightarrow \C$ which corresponding to a decomposable element $v_1\otimes\cdots\otimes v_k\in V_\sigma^{\otimes k}$ is $\mu_1\otimes\cdots\otimes \mu_n\longmapsto \prod \mu_i(v_i)$.
\end{proof}
\end{thm}

By Theorem \ref{Thm: invariants of SO(n)}, we see that, unlike the CM case, if the endomorphism field is totally real, there are Hodge classes in $\bigotimes^{\bullet}T(X)$ which cannot be expressed in terms of Hodge classes of degree two. To describe the additional classes needed to generate $\left(\bigotimes^{\bullet}T(X)\right)^{\mathrm{Hdg}(X)}$, let us recall the following result from linear algebra:
\begin{lem}\cite[Lem.\ 4.3]{deligne1982hodge}
\label{Lem.: wedge and field extensions}
Let $k$ be a field, and let $V$ be a vector space of finite dimension over a finite separable field extension $k'$ of $k$. Then, for any integer $r$, there is a natural embedding of $k$ vector spaces:
\[
\textstyle\bigwedge^r_{k'} V\longhookrightarrow \bigwedge^r_{k}V.
\]
\begin{proof}
For later use, we recall the construction of the embedding.
By Lemma \ref{Lem: Trace}, the trace map induces an isomorphism of $k$-vector spaces
\[\mathrm{Hom}_{k'}(V,k')\simeq \mathrm{Hom}_k(V,k),\quad f\longrightarrow\mathrm{Tr}_{k'/k}\circ f.\]
This induces a natural map
\[
\kappa\colon\textstyle\bigwedge_k^{\bullet}\mathrm{Hom}_k(V,k)\longrightarrow\bigwedge_{k'}^{\bullet}\mathrm{Hom}_{k'}(V,k').
\]
The dual of $\kappa$ as a map of $k$-vector spaces gives a map
\[
\kappa^*\colon\textstyle\mathrm{Hom}_k\left(\bigwedge_{k'}^{\bullet}\mathrm{Hom}_{k'}(V,k'),k\right)\longrightarrow\mathrm{Hom}_k\left(\bigwedge_{k}^{\bullet}\mathrm{Hom}_{k}(V,k),k\right).\]
The desired embedding is then the composition
\begin{align*}
    \textstyle\bigwedge^{\bullet}_{k'} V&\textstyle\stackrel{\simeq}{\longrightarrow} \mathrm{Hom}_{k'}\left(\bigwedge_{k'}^{\bullet}\mathrm{Hom}_{k'}(V,k'),k'\right)\xrightarrow{\mathrm{Tr}_{k'/k}} \mathrm{Hom}_{k}\left(\bigwedge_{k'}^{\bullet}\mathrm{Hom}_{k'}(V,k'),k\right)\\
    &\textstyle\stackrel{\kappa^*}{\longrightarrow}\mathrm{Hom}_k\left(\bigwedge_{k}^{\bullet}\mathrm{Hom}_{k}(V,k),k\right)\stackrel{\simeq}{\longrightarrow}\bigwedge^{\bullet}_{k}V.\qedhere
\end{align*}
\end{proof}
\end{lem}
Let $r\coloneqq \dim_E T(X)=\dim V_\s$. Applying Lemma \ref{Lem.: wedge and field extensions} with $V=T(X)$, $k'=E$, and $k=\Q$, we get the following embedding
\[
\textstyle\bigwedge^r_{E}T(X)\longhookrightarrow \bigwedge^r T(X).
\]
Since the right-hand side embeds naturally into $T(X)^{\otimes r},$ we can consider $\bigwedge^r_{E}T(X)$ as a vector subspace of $T(X)^{\otimes r}.$ Then, using the isomorphism $E\otimes_\Q \C\simeq\prod_{\sigma}\C$, one sees that the $\C$-linear span of this vector subspace is 
\[
\textstyle(\bigwedge^r_{E}T(X))\otimes_\Q\C\simeq\bigwedge^r_{E\otimes_\Q \C}T(X)_\C\simeq \bigoplus_\sigma (\bigwedge^r_\C V_\s) \simeq \langle \det V_\s\rangle_\s,
\]
where $\det V_\s$ are as in Theorem \ref{Thm: invariants of SO(n)}. For a similar computation, see \cite[Prop.\ 4.4]{deligne1982hodge}. 
Since $\det V_\s$ are invariant under the action of $\mathrm{Hdg}(X)(\C)$, we deduce that the vector subspace $\bigwedge^r_{E}T(X)$ consists of Hodge classes, we call them \textit{the exceptional Hodge classes}. 
As a consequence of Theorem \ref{Thm: invariants of SO(n)}, we conclude the following:
\begin{thm}
\label{generators in the case of totally real K3}
Let $X$ be a K3 surface with totally real endomorphism field $E$. Then, any Hodge class in $\bigotimes^{\bullet}T(X)$ can be expressed in terms of Hodge classes of degree two and the exceptional Hodge classes in $T(X)^{\otimes r},$ where $r\coloneqq \dim_E T(X)$.
\end{thm}

The name \textit{exceptional classes} alludes to the fact that those classes cannot be expressed in terms of Hodge classes in the square of the transcendental lattice. This is a consequence of the following fact: Let $V_\s$ be as in Theorem \ref{Thm: invariants of SO(n)}. The $\mathrm{SO}(V_\s,\C)$-invariant elements in $V_\s\otimes V_\s$ are also $\mathrm{O}(V_\s,\C)$-invariant, as they are complete contractions. On the other hand, the class $\det V_\s$ is just $\mathrm{SO}(V_\s,\C)$-invariant and thus cannot be expressed in terms of degree-two elements. In the literature, the name \textit{exceptional classes} is used in the context of abelian varieties and denotes Hodge classes which are not in the algebra generated by (intersection of) divisor classes.
\begin{rmk}
The mistake in the proof of \cite[Prop.\ 5.2]{mari2008hodge} in the case of K3 surfaces with totally real multiplication is the wrong assumption that any invariant under the action of the special orthogonal group can be expressed in terms of invariants of degree two.
\end{rmk}

In conclusion, if the endomorphism field $E$ of the K3 surface is totally real, then there are $|E:\Q|$ Hodge classes in $T(X)^{\otimes r}$, with $r=\dim_E T(X)$, which cannot be expressed in terms of Hodge classes in $T(X)^{\otimes 2}$. Thus, in this case, it is not true that the Hodge conjecture for the second power of the K3 surface implies the Hodge conjecture for all its powers.
\subsection{A motivic approach}
Inspired by the techniques in the paper \cite{milne2021tate}, we give a different proof of the results of the previous section.

\smallskip

Let $X$ be a K3 surface and denote by $E$ its endomorphism field. Similarly to the case of abelian varieties, we give the following definition:
\begin{defn}
The \textit{Lefschetz group} of $X$ is the biggest algebraic subgroup of $\mathrm{GL}(T(X))$ which preserves the Hodge classes in $T(X)\otimes T(X)$. We denote it by $\mathrm{L}(X)$.
\end{defn}
Assume that $E$ is a CM field and let $F$ be its maximal totally real subfield. Denoting by $\varphi\colon T(X)\times T(X)\longrightarrow E$ the $E$-Hermitian map of Lemma \ref{Lem: varphi CM K3 surface}, we see that  
\[L(X)\simeq \mathrm{Res}_{F/\Q} \mathrm{U}(T(X), \varphi),\] where we regard $\mathrm{U}(T(X), \varphi)$ as an algebraic group over $\Q$ using the Weil restriction. Note that in this case, $\mathrm{L}(X)$ and $\mathrm{Hdg}(X)$ coincide.
Similarly, if $E$ is a totally real field and $\varphi\colon T(X)\times T(X)\longrightarrow E$ is the $E$-bilinear form on $T(X)$, the Lefschetz group of $X$ is \[\mathrm{L}(X)\simeq \mathrm{Res}_{E/\Q} \mathrm{O}(T(X), \varphi).\]
In particular, by the invariant theory results of previous section, the algebra
$
\textstyle\left(\bigotimes^\bullet T(X)\right)^{\mathrm{L}(X)}
$
is generated by its degree two elements both in the totally real and in the CM case.
\begin{defn}
The \textit{motivic group} $\mathrm{M}(X)$ is the biggest algebraic subgroup of $\mathrm{GL}(T(X))$ which preserves all the algebraic classes in $\bigotimes^\bullet T(X)$.
\end{defn}
In this case, the fact that only the algebraic classes in $\bigotimes^\bullet T(X)$ are invariant under the action of ${\mathrm{M}(X)}$ is non-trivial: As in the case of abelian varieties \cite[App. A]{milne2021tate}, it follows from the fact that the category of motives on a variety is a Tannakian category if the K\"{u}nneth components of the diagonal are algebraic \cite[Cor.\ 3]{jannsen1992motives}.

\smallskip
Since all algebraic classes are Hodge classes, there is an inclusion
$\mathrm{Hdg}(X)\subseteq \mathrm{M}(X)$. Then, proving the Hodge conjecture for the powers of $X$ is equivalent to prove that this inclusion is an equality.

\smallskip

If we assume the Hodge conjecture for $X^2$, i.e., that the classes in $E\simeq \left(T(X)\otimes T(X)\right)^{\mathrm{Hdg}(X)}$ are algebraic, we have the following chain of inclusions:
\begin{equation}
    \label{eq inclusions of 3 groups}
    \mathrm{Hdg}(X)\subseteq \mathrm{M}(X)\subseteq \mathrm{L}(X).
\end{equation}
In the case of K3 surfaces with CM endomorphism field, the Hodge group and the Lefschetz group are equal, all the inclusions in $(\ref{eq inclusions of 3 groups})$ are therefore equalities. This gives another proof of Theorem \ref{Thm: invariant degree 2 elements generate in the CM case}, and shows that, in the case of CM endomorphism field, the Hodge conjecture for $X^2$ implies the Hodge conjecture for all powers of $X$.

\smallskip

As we saw in the previous section, in the case of K3 surfaces with a totally real endomorphism field, the situation is different: Even if we assume the Hodge conjecture for the square we cannot conclude the Hodge conjecture for all powers of the given K3 surface due to the presence of the exceptional Hodge classes. From the group perspective, this is a consequence of the fact  
that the Hodge group and the Lefschetz group are not equal, so $(\ref{eq inclusions of 3 groups})$ does not imply the equality between the Hodge and the motivic group. Assuming the Hodge conjecture for the square of the K3 surface, let us find which are the possibilities for the motivic group:
Let $X$ be a K3 surface with totally real endomorphism field $E$ and let us take the complex points in $(\ref{eq inclusions of 3 groups})$
\[
\prod_{\sigma\colon E\hookrightarrow \C}\mathrm{SO}(V_\s,\C)
\subseteq \mathrm{M}(X)(\C)\subseteq
\prod_{\sigma\colon E\hookrightarrow \C}\mathrm{O}(V_\s,\C).
\]
We see that $\mathrm{M}(X)(\C)$ corresponds to a subgroup of the quotient 
\[
\left(\prod_{\sigma\colon E\hookrightarrow \C}\mathrm{O}(V_\s,\C)\right)/\left(\prod_{\sigma\colon E\hookrightarrow \C}\mathrm{SO}(V_\s,\C)\right)\simeq (\mathbb{Z}/2\mathbb{Z})^{|E:\Q|}.
\]
From the fact that $\mathrm{M}(X)$ is defined over $\Q$, there are constraints on the possibilities for $\mathrm{M}(X)(\C)$:
Let $\Bar{E}$ be the Galois closure of the field $E$ and let $G\coloneqq \mathrm{Gal}(\Bar{E}/\Q)$ be the Galois group of $\Bar{E}$ over $\Q$, then $G$ acts on $
(\mathbb{Z}/2\mathbb{Z})^{|E:\Q|}$ by changing the factors.
As $\mathrm{M}(X)$ is defined over $\Q$, the group $\mathrm{M}(X)(\C)$ corresponds to a subgroup of $(\mathbb{Z}/2\mathbb{Z})^{|E:\Q|}$ which is preserved under the $G$-action. 
This allows us to conclude that 
\[\textstyle\mathrm{M}(X)(\C)/\left(\prod
_{\sigma}\mathrm{SO}(V_\s,\C)\right)\] is either trivial, the full group $(\mathbb{Z}/2\mathbb{Z})^{|E:\Q|}$, or the subgroup of $(\mathbb{Z}/2\mathbb{Z})^{|E:\Q|}$ generated by $(1,\ldots,1)$. 

\smallskip

In the first case, $\mathrm{M}(X)$ coincides with $\mathrm{Hdg}(X)$ and hence the Hodge conjecture holds for all powers of $X$. In the second case, $\mathrm{M}(X)$ coincides with $\mathrm{L}(X)$ and only
the classes in the algebra generated by the Hodge classes in $T(X)\otimes T(X)$ are algebraic. Finally, in the latter case, 
\[
\textstyle\mathrm{M}(X)(\C)=\{A_1\times\cdots \times A_k\in \prod_\s \mathrm{O}(V_\s,\C)|\prod_i\det(A_i)=1\}.
\]
This corresponds to the case where the class $\det T(X)$ is algebraic but the exceptional classes in $\bigwedge^r_E T(X)$ are not.
Showing the Hodge conjecture for all powers of $X$ is then equivalent to exclude the two latter cases. In particular, note that to prove the Hodge conjecture for all powers of a K3 surface $X$ it suffices to prove it for $X^2$ and then show that there exists an algebraic class in in $\bigwedge^r_E T(X)$. Indeed, this would force the motivic group to be equal to the Hodge group by the above discussion.
\begin{rmk}
The fact that we can exclude some possibilities for $\mathrm{M}(X)(\C)$ 
is linked with the fact that $\left(\bigotimes^\bullet T(X)_\C\right)^{\mathrm{M}(X)(\C)}$ is defined over $\Q$: For example, assume that $\mathrm{M}(X)(\C)$ corresponds to the subgroup of $(\mathbb{Z}/2\mathbb{Z})^{|E:\Q|}$ generated by $(0,1,\ldots,1)$, i.e. \[\mathrm{M}(X)(\C)\simeq \mathrm{SO}(V_{\sigma_1},\C)\times \mathrm{O}(V_{\sigma_2},\C)\times \cdots \times \mathrm{O}(V_{\sigma_k},\C).\] 
Then $\left(\bigwedge^r_E T(X)\right)^{\mathrm{M}(X)}\otimes_\Q\C$ is one-dimensional and generated by $\det V_{\s_1}$. In particular, up to complex scalar, $\det V_{\s_1}$ is a rational cohomology class. To see that this cannot happen, consider the action of $E$ on the first factor of $T(X)^{\otimes r}$. As any element of $E$ acts on $\det V_{\s_1}$ by multiplication by its image via $\sigma_1\colon E\longhookrightarrow \C$, the class $\det V_{\s_1}$ cannot be rational.
\end{rmk}
\begin{rmk}
Note that this motivic approach is equivalent to the approach we presented. It might seem easier due to the fact that we implicitly used in more places the results proved above. For example, to show that the algebra of invariant classes under the Lefschetz group is generated by degree-two elements, we used the invariant theory of the previous section.
\end{rmk}

\subsection{The Hodge conjecture and deformations of K3 surfaces}
Let $X$ be a K3 surface and let $n\coloneqq \dim_\Q T(X)$. Let $\det T(X)$ be a generator of the one-dimensional vector space
\[
\textstyle
\bigwedge^{n}T(X)\subseteq T(X)^{\otimes n}
\]
Note that $\det T(X)$ is a Hodge class for every K3 surface $X$. 
If the endomorphism field of $X$ is a CM field, by Theorem \ref{Thm: invariant degree 2 elements generate in the CM case}, the class $\det T(X)$ can be expressed in terms of $(2,2)$-classes. On the other hand, if the endomorphism field is $\Q$ or any other totally real field, $\det T(X)$ cannot be expressed in terms of Hodge classes of degree two.

\smallskip

In this section, we prove that the property \textit{the determinant of the transcendental lattice is algebraic} is a closed property in families of K3 surfaces:

\begin{prop}
\label{Prop.: determinant is algebraic is a closed property in families}
Let $\X\longrightarrow S$ be a family of K3 surfaces such that $\det T(\X_s)$ is algebraic for general $s\in S$. Then, the same holds for all $s\in S$.
\begin{proof}
We may assume that the transcendental lattices of the general fibres of $\X\longrightarrow S$ are isometric to a given quadratic space $\Tilde{T}$. Write $n\coloneqq \dim_\Q \Tilde{T}$. Since $\det T(\X_s)$ is algebraic for general $s\in S$ by assumption, we conclude that $\det \T$ is algebraic when seen as a class of $H^{r,r}(\X_s^r,\Q)$ for all $s\in S$. Note that this does not conclude the proof, since $\det \Tilde{T}$ does not necessarily specialize to $\det T(\X_s)$ for every $s\in S$, since it may happen that $\dim T(\X_s)<n$ for some $s\in S$ . 
Fix an element $0\in S$ and let $X\coloneqq \X_0$ be the corresponding K3 surface. By construction, the transcendental lattice of $X$ is naturally a subspace of $\T$. Hence, if $\dim_\Q T(X)=n$ then $T(X)=\T$ (as rational quadratic spaces) and $\det T(X)$ is algebraic so there is nothing to prove. 
Let assume that $\dim_\Q \T-\dim_\Q T(X)=c>0$ and denote by $\T/T(X)$ the quotient vector space.
Up to a rational coefficient, we have the following equality:
\begin{equation}
    \label{eq.: determinant and quotient}
    \det \T=\sum\pm q_I^*\det T(X)\otimes p_I^* \det (\T/T(X)),
\end{equation}
where the sum runs over all subsets $I\subseteq \{1,\ldots,n\}$ of length $c$, $p_I$ is the projection from $X^n$ onto the $I$-th factors and $q_I$ is the projection onto the remaining factors.

\smallskip

Let $Z_1$ be an algebraic class representing $\det \T$ on $X^n$.
Then, $Z_1$ defines the following algebraic map:
\begin{align*}
    \varphi_1\colon H^{1,1}(X,\Q) \longrightarrow & H^{n-1,n-1}(X^{n-1},\Q),\\
    y \longmapsto & q_{1*}(Z_1\cap p_1^*y)
\end{align*}
where $p_1$ is the projection from $X^n$ onto the first factor and $q_1$ is the projection onto the remaining factors. 
Let $x_1\in H^{1,1}(X,\Q)$ be a fixed element and denote by $Z_2$ the algebraic class $\varphi_1(x_1)\in H^{n-1,n-1}(X^{n-1},\Q)$. Then, similarly to the previous case, $Z_2$ defines an algebraic map
\begin{align*}
    \varphi_2\colon H^{1,1}(X,\Q) \longrightarrow & H^{n-2,n-2}(X^{n-2},\Q).\\
    y \longmapsto & q_{2*}(Z_1\cap p_2^*y)
\end{align*}
Repeating this procedure $c$-times we end up with an algebraic class $Z_c$ in $H^{n-c,n-c}(X^{n-c},\Q).$ Note that the only summand of (\ref{eq.: determinant and quotient}) contributing to the resulting class $Z_c$ is the one for which $I=\{1,\ldots,c\}$. This follows from the fact that all the elements of $T(X)$ are orthogonal to $\mathrm{NS}(X)$ with respect to the canonical pairing on $X$. 
Using a basis of the rational quotient space $\T/T(X)$ consisting of algebraic classes, we see that it is possible to choose a sequence $x_1,\ldots,x_c\in H^{1,1}(X,\Q)$ for which the resulting $Z_c$ is non-zero and represents the class $\det T(X)\in H^{n-c,n-c}(X^{n-c},\Q).$ This shows that $\det T(X)$ is algebraic and concludes the proof.
\end{proof}
\end{prop}
As an immediate corollary, we have the following:
\begin{cor}
\label{Cor.: when the determinant is the only class to study}
Let $\X\longrightarrow S$ be a family of K3 surfaces such that the Hodge conjecture for all powers of $\X_s$ holds  for general $s\in S$. If $0\in S$ is an element such that $\mathrm{End}_{\mathrm{Hdg}(\X_0)}(T(\X_0))=\Q,$ then the Hodge conjecture holds for all powers of $\X_0$.
\begin{proof}
Since we are assuming that the endomorphism field of $\X_0$ is $\Q$, by Theorem \ref{generators in the case of totally real K3}, the algebra of Hodge classes in $\bigotimes^{\bullet} T(\X_0)$ is generated by
\[\mathrm{End}_{\mathrm{Hdg}(\X_0)}(T(\X_0))\simeq \Q\; \text{ and } \det T(\X_0).\]
This implies the statement, indeed, the generator of $\mathrm{End}_{\mathrm{Hdg}(\X_0)}(T(\X_0))$ is algebraic, since it corresponds to a component of the class of the diagonal $\Delta\subseteq \X_0\times \X_0$, and $\det T(\X_0)$ is algebraic by Proposition \ref{Prop.: determinant is algebraic is a closed property in families}.
\end{proof}
\end{cor}
\begin{rmk}
With a little bit of work, one can prove the same statement of Proposition \ref{Prop.: determinant is algebraic is a closed property in families} for the exceptional Hodge classes. Note that the determinant of the transcendental lattice is the unique exceptional Hodge class in the case $E=\Q$.
This in particular shows that, in the statement of Corollary \ref{Cor.: when the determinant is the only class to study}, the assumption $\mathrm{End}_{\mathrm{Hdg}(\X_0)}(T(\X_0))=\Q$ can be weakened to 
\[\mathrm{End}_{\mathrm{Hdg}(\X_0)}(T(\X_0))=\mathrm{End}_{\mathrm{Hdg}(\X_s)}(T(\X_s)).\]
Note that the inclusion \textquotedblleft $\supseteq$" always holds true. We do not give here the detailed proof since this extended result is not needed in the remainder of this paper.
\end{rmk}

\section{Abelian varieties of Weil type}
\label{Sec. Ab. of Weil type}
In this section, we recall the definition of abelian varieties of Weil type and we study the Hodge conjecture for their powers. Our interest in these varieties lies in the fact that, as we recall in the next section, the Kuga--Satake varieties of K3 surfaces of Picard number $16$ are powers of abelian fourfolds of Weil type. These abelian varieties have been first introduced and studied by Weil \cite{weil}. We refer to \cite{van1994introduction} for a complete introduction. For a survey on the Hodge conjecture for abelian varieties we refer to \cite{moonen1995hodge} and to \cite{moonen1999hodge}.

\smallskip

In this section, we denote by $K$ the CM field $\Q(\sqrt{-d})$, where $d$ is a positive rational number.

\begin{defn}
Let $A$ be an abelian variety of dimension $2n$ such that $K\subseteq \mathrm{End}(A)\otimes_\mathbb{Z}\Q$.
Then, $A$ is an abelian variety of $K$\textit{-Weil type} if the action of $\sqrt{-d}$ on the tangent space at the origin of $A$ has eigenvalues $\sqrt{-d}$ and $-\sqrt{-d}$ both with multiplicity $n$.
\end{defn}

Given an abelian variety $A$ of $K$-Weil type, there exists a polarization $H$ on $A$ such that $(\sqrt{-d})^* H= d H$. This polarization induces the following $K$-Hermitian map on $H^1(A,\Q)$:
\begin{equation}
\label{Eq.: Hermitian form abelian varieties of Weil type}
   \tilde{H}\colon H^1(A,\Q)\times H^1(A,\Q) \longrightarrow K, \hspace{0.25cm} \tilde{H}(x,y)\coloneqq H(x,\sqrt{-d}y)+\sqrt{-d}H(x,y).
\end{equation}
Let $\mathrm{N}(K)$ be the set of norms of $K$, and let $\mathrm{det}(\tilde{H})$ be the determinant of $\tilde{H}$ with respect to any $K$-basis of $H^1(A,\Q)$. Then, one checks that $\det(\tilde{H})$ belongs to $\Q^*$ and that its image $\delta$ in $\Q/\mathrm{N}(K)$ does not depend on the choice of the basis. The class $\delta$ is called the \textit{discriminant} of $A$. We say that an abelian variety $A$ of $K$-Weil type is \textit{general} if it has maximal Hodge group. By \cite[Thm.\ 6.11]{van1994introduction}, this means that $\mathrm{Hdg}(A)=\mathrm{SU}(\tilde{H})$ and that $\mathrm{End}(A)\otimes_\mathbb{Z}\Q=K$.

\smallskip
Note that there are similarities between this situation and the case of K3 surfaces with CM endomorphism field.
Let $E$ be a CM field with maximal totally real subfield $\Q$, and let $X$ be a K3 surface with endomorphism field $E$ and let $2n\coloneqq \dim T(X)$. By \cite[Thm.\ 2.3.1]{zarhin1983hodge}, the Hodge group of $X$ is $U(\varphi)$, where $\varphi$ is $E$-Hermitian form on $T(X)$. The discussion of Section \ref{Section: CM case} and in particular Lemma \ref{Lem: the action of S} show that $U(\varphi)(\C)\simeq \mathrm{GL}(n,\C)$ and that the representation of $\mathrm{Hdg}(X)(\C)$ on $T(X)_\C$ is the direct sum of the standard representation of $\mathrm{GL}(n,\C)$ and its dual. 
In the case of a general abelian variety $A$ of Weil type of dimension $2n$, the Hodge group is $\mathrm{SU}(\tilde{H})$. Then, a similar argument shows the following:
\begin{prop}\cite[Lem.\ 6.10]{van1994introduction}
\label{Prop.: Hodge group of a general abelian variety of Weyl type}
Let $A$ be a general abelian variety of Weil type of dimension $2n$. Then, $\mathrm{Hdg}(A)(\C)$ is isomorphic to $\mathrm{SL}(2n,\C)$, and the representation $H^1(A,\C)$ of $\mathrm{Hdg}(A)(\C)$ is the direct sum of the standard representation of $\mathrm{SL}(2n,\C)$ and its dual.
\end{prop}

Abelian varieties of Weil type are characterized by the existence of Hodge classes which do not belong to the algebra generated by divisors:
\begin{rmk}
\label{Rmk.: exceptional classes on A}
Let $A$ be an abelian variety of $K$-Weil type of dimension $2n$. By Lemma \ref{Lem.: wedge and field extensions}, there exists an embedding
\[
\textstyle\epsilon\colon\bigwedge_{K}^{2n} H^1(A,\Q)\longhookrightarrow \bigwedge^{2n} H^1(A,\Q)\simeq H^{2n}(A,\Q).
\]
The condition on the action of the field $K$ on $H^1(A,\Q)$ implies that the image of $\epsilon$ consists of Hodge classes, see \cite[Prop.\ 4.4]{deligne1982hodge}. These classes are called \textit{exceptional classes} or \textit{Weil classes}. From now on, we identify $\bigwedge_{K}^{2n} H^1(A,\Q)$ with its image via $\epsilon$ in $\bigwedge^{2n} H^1(A,\Q)$.
\end{rmk}

By Proposition \ref{Prop.: Hodge group of a general abelian variety of Weyl type}, the algebra of Hodge classes on a general abelian variety of Weil type satisfies the following
\[
\textstyle\left(\bigwedge^{\bullet}H^1(A,\Q)\right)^{\mathrm{Hdg}(A)}\otimes \C\simeq \left(\bigwedge^{\bullet} (W\oplus W^*)\right)^{\mathrm{SL}(W,\C)},
\]
where $W$ is a complex vector space of dimension $2n$. Complete contractions are natural examples of invariant elements of the right-hand side. Let us recall their definitions. 
\begin{defn}
Let $W$ be a $2n$-dimensional vector space and let $s$ be a positive integer. Then, an element in $W^{\otimes s}\otimes (W^*)^{\otimes s}$ is called a \textit{complete contraction} of degree $2s$ if its image under the natural $\mathrm{SL}(W,\C)$-invariant isomorphism 
\[W^{\otimes s}\otimes (W^*)^{\otimes s}\simeq
\mathrm{Hom}((W^*)^{\otimes s}\otimes W^{\otimes s},\C)
\] 
is equal to 
\[
\textstyle\mu_1\otimes\cdots\otimes\mu_s\otimes v_1\otimes\cdots\otimes v_s\longmapsto \prod_i \mu_i(v_{\sigma(i)}),
\]
for some permutation $\sigma\in \mathfrak{S}^s$
\end{defn}

\begin{rmk}
\label{rmk.: complete contractions and degree-two complete contractions}
From the definition, it is immediate to see that complete contractions are invariant under the action of $\mathrm{SL}(W,\C)$ and that any complete contraction can be expressed in terms of complete contractions of degree two.
\end{rmk}
Let $s$ and $s'$ be non-negative integers, and let $I=(i_1,\ldots,i_k)$ and $I'=(i_1',\ldots,i_k')$ be partitions of $s$ and $s'$ respectively, i.e., $i_1+\cdots +i_k=s$ and $i'_1+\cdots +i'_k=s'$. Considering the natural action of $\mathfrak{S}^{s}\times \mathfrak{S}^{s'}$ on
$W^{\otimes s}\otimes(W^*)^{\otimes s'}$, we introduce following definition.
\begin{defn}
\label{defn: I-alternating contractions}
An element  $\alpha\in W^{\otimes s}\otimes(W^*)^{\otimes s'}$ is $(I,I')$\textit{-alternating} if it satisfies 
\[(\sigma,\sigma')(\alpha)=\text{sgn}(\sigma)\text{sgn}(\sigma')\alpha,\]
for any pair of permutations $\sigma\in\mathfrak{S}^{i_1}\times\cdots\times \mathfrak{S}^{i_k}\subseteq\mathfrak{S}^{s}$ and $\sigma'\in\mathfrak{S}^{i_1'}\times\cdots\times \mathfrak{S}^{i_k'}\subseteq\mathfrak{S}^{s'}$.
\end{defn}
In the case where $s=s'$ and $I=I'=(s)$, the vector space of $(I,I')$-alternating elements in $W^{\otimes s}\otimes(W^*)^{\otimes s}$ is equal to $\bigwedge^s W\otimes \bigwedge^s W^*$.

\smallskip

We are now able to state the following theorem from invariant theory.
\begin{thm}\cite[Thm.\ 8.4]{kraft1996classical}
\label{Thm: invariants under SL}
Let $W$ be a $2n$-dimensional complex vector space and denote by $W^*$ its dual. Then, any element of 
\[
\textstyle\left(\bigotimes^{\bullet}(W\oplus W^*)\right)^{\mathrm{SL}(W,\C)}
\]
can be expressed in terms of complete contractions and the two determinants $\det W$ and $\det W^*$ in $(W\oplus W^*)^{\otimes 2n}$.
\begin{proof}
See Definition \ref{def.: meaning of expressed in terms of} for formal meaning of the expression \textquotedblleft can be expressed in terms of". To deduce this statement from the one in the reference, just note that, for any pair of non-negative integers $p$ and $q$, any element of $W^{\otimes p}\otimes (W^*)^{\otimes q}\subseteq(W\oplus W^*)^{\otimes (p+q)}$ can be viewed as a homogeneous polynomial function $(W^*)^{\oplus p}\oplus W^{\oplus q}\longrightarrow \C$ of degree $p+q$.
\end{proof}
\end{thm}
Applying Theorem \ref{Thm: invariants under SL} to a general abelian variety of Weil type, one can compute the dimension of the vector space of Hodge classes. Indeed, the following holds:

\begin{thm}\cite[Thm.\ 6.12]{weil}
\label{Thm.: dimension of Hodge ring for general abelian varieties of Weil type}
Let $A$ be a general abelian variety of $K$-Weil type of dimension $2n$. Then, 
\[
\dim(H^{s,s}(A,\Q))=
\begin{cases} 3, & \mbox{if }s=n \\ 1, & \mbox{if }s \not =n.
\end{cases}
\]
In particular, the Picard number of $A$ is one.
\begin{proof}
We just give a sketch of the proof since the same techniques will be used with much more detail later in this section: By assumption, $A$ is a general abelian variety of Weil type. Then, as before, we have the following isomorphism:
\begin{equation}
    \label{eq. a}
\textstyle\left(\bigwedge^{\bullet}H^1(A,\Q)\right)^{\mathrm{Hdg}(A)}\otimes_\Q\C\simeq
\left(\bigwedge^{\bullet}(W\oplus W^*)\right)^{\mathrm{SL}(W,\C)},
\end{equation}
where $W$ is a $2n$-dimensional complex vector space.
After considering the natural $\mathrm{SL}(W,\C)$-invariant embedding \[\textstyle\bigwedge^{\bullet}(W\oplus W^*)\longhookrightarrow\bigotimes^{\bullet}(W\oplus W^*),\]
we can apply Theorem \ref{Thm: invariants under SL} to deduce that the $\C$-algebra $\left(\bigwedge^{\bullet}(W\oplus W^*)\right)^{\mathrm{SL}(W,\C)}$ is generated by $(s,s)$-alternating linear combinations of complete contractions in $\bigotimes^{2s}(W\oplus W^*)$ for all $1\leq s\leq 2n$ together with the two determinants $\det W, \det W^*\in \bigotimes^{2n}(W\oplus W^*)$. 
Note that for every $s$ there exists a unique linear combination of complete contractions in $W^{\otimes s}\otimes (W^*)^{\otimes s}$ which is $(s,s)$-alternating: It corresponds to the $\mathrm{SL}(W)$-invariant map
\begin{align*}
  \textstyle  \bigwedge^s W^*\otimes \bigwedge^s W&\longrightarrow\C\\
   \mu_1\wedge\cdots\wedge\mu_s\otimes v_1\wedge\cdots\wedge v_s&\longmapsto \det (\mu_i(v_j))_{i,j}
\end{align*} 
This, together with the isomorphism in (\ref{eq. a}), proves the statement.
\end{proof}
\end{thm}
\begin{rmk}
\label{rmk.: Hodge classes abelian var of weil type}
As an immediate consequence of Theorem \ref{Thm.: dimension of Hodge ring for general abelian varieties of Weil type}, we deduce that the algebra of Hodge classes on $A$ is generated by the unique class in $\mathrm{NS}(A)_\Q$ and the exceptional classes in $\bigwedge^{2n}_{K}H^1(A,\Q)$. As in the case of K3 surfaces with totally real multiplication, this follows from the isomorphism
\[
\textstyle
\left( \bigwedge^{2n}_{K}H^1(A,\Q) \right)\otimes_\Q \C\simeq\langle \det W,\det W^*\rangle_\C,
\]
cf. \cite[Prop.\ 4.4]{deligne1982hodge}.
\end{rmk}
In general, it is not known whether the Weil classes are algebraic, but in the case of abelian fourfolds of Weil type with discriminant one, the following holds:
\begin{thm}\cite[Thm.\ 1.5]{markman2022monodromy}
\label{Thm.: Markman}
Let $A$ be abelian fourfold of Weil type with discriminant one. Then, the Weil classes on $A$ are algebraic. In particular, the Hodge conjecture holds for the general abelian fourfold of Weil type with discriminant one.
\end{thm}
The same result was previously proven by Schoen \cite{schoen1988hodge} in the case of abelian fourfolds $\Q(\sqrt{-3})$-Weil type for arbitrary discriminant and by Schoen \cite{schoen1988hodge} and independently by van Geemen \cite{van1996theta} in the case of abelian fourfolds $\Q(i)$-Weil type with discriminant one.

\smallskip

In the remainder of this section, we prove that the Hodge conjecture for a general abelian variety of Weil type implies the Hodge conjecture for all its powers.
Our strategy is the following: We first prove an extension of Theorem \ref{Thm: invariants under SL} which allows us to find a set of generators for the algebra of Hodge classes on the powers of $A$. Then, we show that there are relations between these generators and we conclude that these generators are algebraic if the Hodge conjecture holds for $A$. For a comparison with the work of Abdulali \cite{abdulali1999abelianIII}, see Remark \ref{rmk.: comparison with Abdulali}.

\smallskip

Let $A$ be a general abelian variety of Weil type and let $k$ be a positive integer. By Theorem \ref{Thm: invariants under SL}, the ring of Hodge classes on $A^k$ satisfies
\[
\textstyle\left(\bigwedge^{\bullet}(H^1(A,\Q)^{\oplus k})\right)^{\mathrm{Hdg}(A)}\otimes\C\simeq \left(\bigwedge^{\bullet}((W\oplus W^*)^{\oplus k})\right)^{\mathrm{SL}(W,\C)},
\]
where $W$ is a $2n$-dimensional complex vector space. To study this ring, let us introduce the notion of \textit{realizations} of $\det W$ and $\det W^*$.

\begin{rmk}
\label{rmk.: all possible realizations of determinants}
Let $W$ be a $2n$-dimensional complex vector space and let $k$ be a positive integer. Consider the following canonical decomposition:
\[
\textstyle
\bigwedge^{2n}\left((W\oplus W^*)^{\oplus k}\right)=\bigoplus \left( \bigwedge^{i_1}W_1\otimes\cdots\otimes \bigwedge^{i_k}W_k\otimes \bigwedge^{i_{k+1}}W_1^*\otimes\cdots\otimes \bigwedge^{i_{2k}}W_k^*
\right),
\]
where the sum runs over all $2k$-partitions $I$ of $2n$ and $W_j=W$ for all $j$. We introduced the $W_j$ to be able to distinguish between 
$\bigwedge^{2n}W_1$ and $\bigwedge^{2n}W_2$, etc. 
Note that if $I$ is a $k$-partition of $2n$, i.e., $i_{k+1}=\ldots= i_{2n}=0$, there is a natural embedding
\[
\textstyle\iota_{I}\colon\bigwedge^{2n}W\lhook\joinrel\longrightarrow \bigwedge^{i_1}W_1\otimes\cdots\otimes \bigwedge^{i_k}W_k.
\]
This follows from the fact that the image of $\bigwedge^{2n}W\longhookrightarrow W^{\bigotimes 2n}$ is contained in the image of the natural embedding $\bigwedge^{i_1}W_1\otimes\cdots\otimes \bigwedge^{i_k}W_k \longhookrightarrow W^{\bigotimes 2n}$. 
As one sees, $\iota_I$ is compatible with the natural action of $\mathrm{SL}(W,\C)$. Hence, its image determines (up to a complex scalar) an $\mathrm{SL}(W,\C)$-invariant class. Denote this class as $(\det W)_I$ and consider it as an element in $\bigwedge^{2n}\left((W\oplus W^*)^{\oplus k}\right)$. We call it a \textit{realization} of $\det W$. 
We then call
\[
\{\textstyle(\det W)_I\;|\; I\text{ is a } k\text{-partition of } 2n\}\subseteq \bigwedge^{2n}\left((W\oplus W^*)^{\oplus k}\right)
\]
the \textit{set of all realizations} of $\det W$ in $\bigwedge^{2n}\left((W\oplus W^*)^{\oplus k}\right)$. 
Similarly, we introduce the notion of \textit{realizations} of $\det W^*$.

\smallskip

To give an example, let $k=2$, $2n=4$, and consider the $2$-partition of $4$ given by $I=(2,2)$. Then, denoting by $v_1,\ldots, v_4$ a basis of $W$, the realization $(\det W)_I$ in $\bigwedge^{4}\left((W\oplus W^*)\right)^{\oplus 2}$ is given by the image of
\begin{align*}
 \textstyle\bigwedge^4W &\textstyle\longhookrightarrow \bigwedge^2W_1\otimes\bigwedge^2W_2\\
  v_1\wedge\cdots\wedge v_4&\longmapsto \sum\pm (v_i\wedge v_j)\otimes (v_k\wedge v_l),
\end{align*}
where the sum runs over all $i<j,k<l$ such that $\{i,j,k,l\}=\{1,2,3,4\}$.
\end{rmk}

With this notion, we can  now state and prove the following extension of Theorem \ref{Thm: invariants under SL}.

\begin{cor}
\label{Cor: invariants under SL extended version}
Let $W$ be a $2n$-dimensional complex vector space and denote by $W^*$ its dual. Then, for every positive integer $k$, any invariant in
\[
\textstyle\left(\bigwedge^{\bullet}((W\oplus W^*)^{\oplus k})\right)^{\mathrm{SL}(W,\C)}
\]
can be expressed in terms of invariants of degree two and all realizations of $\det W$ and $\det W^*$ in $\bigwedge^{*}\left((W\oplus W^*)^{\oplus k}\right)$ defined in Remark \ref{rmk.: all possible realizations of determinants}.
\begin{proof}
Let $s$ and $k$ be two positive integers. As in Remark \ref{rmk.: all possible realizations of determinants}, we decompose the space $\bigwedge^s\left((W\oplus W^*)^{\oplus k}\right)$ as
\[
\textstyle
\bigwedge^{s} \left( (W\oplus W^*)^{\oplus k} \right)=\bigoplus \left( \bigwedge^{i_1}W_1\otimes\cdots\otimes \bigwedge^{i_k}W_k\otimes \bigwedge^{i_{k+1}}W_1^*\otimes\cdots\otimes \bigwedge^{i_{2k}}W_k^*
\right),
\]
where the direct sum runs over all $2k$-partitions of $s$, and $W_j=W$ for all $j$. For any $2k$-partition $I$ of $s$, denote by $Z_{I,s}$ the corresponding direct summand of the decomposition.
Since the action of $\mathrm{SL}(W,\C)$ preserves this decomposition, an element in $\bigwedge^s\left((W\oplus W^*)^{\oplus k}\right)$ is $\mathrm{SL}(W,\C)$-invariant if and only if its component in $Z_{I,s}$ is invariant for every $2k$-partition $I$ of $s$. Therefore, it suffices to study the invariants in each $Z_{I,s}$.
To this end, we first embed each $Z_{I,s}$ individually into $\bigotimes^s(W\oplus W^*)$ and then apply Theorem \ref{Thm: invariants under SL}:

\smallskip

Let $I=(i_1,\ldots, i_{2k})$ be a fixed $2k$-partition of $s$. Let $s'\coloneqq i_1+\cdots+i_k$, and consider the natural $\mathrm{SL}(W,\C)$-invariant embedding of $ Z_{I,s}$
\begin{equation}
    \label{eq.: embedding of Z_I,s}
    \textstyle
Z_{I,s}=\bigwedge^{i_1}W_1\otimes\cdots\otimes \bigwedge^{i_k}W_k\otimes \bigwedge^{i_{k+1}}W_1^*\otimes\cdots\otimes \bigwedge^{i_{2k}}W_k^*\longhookrightarrow W^{\otimes s'}\otimes(W^*)^{\otimes (s-s')}
\end{equation}
given by the tensor product of the canonical embeddings 
\[
\textstyle
\bigwedge^{i_j}W\longhookrightarrow W^{\otimes i_j} \;\text{and}\; \bigwedge^{i_j}W^*\longhookrightarrow (W^*)^{\otimes i_j}.
\]
From now on, we identify $Z_{I,s}$ and its image. Note that an element of $W^{\otimes s'}\otimes(W^*)^{\otimes (s-s')}$ belongs to $Z_{I,s}$ if and only if it is $I$-alternating.

\smallskip

Let us first deal with the case where $s<2n$: Since $\det W$ and $\det W^*$ are elements of $(W\oplus W^*)^{\otimes 2n}$, Theorem \ref{Thm: invariants under SL} shows that the vector space of invariants in $W^{\otimes s'}\otimes (W^*)^{\otimes s-s'}$ 
is zero if $s-s'\not =s'$, and it is generated by complete contractions if $s-s'=s'$.
Therefore, we conclude that $Z_{I,s}$ does not contain any non-trivial invariant if $s-s'\not= s'$ and it is generated by linear combinations of complete contractions which are $I$-alternating if $s-s'=s'$.

\smallskip

For $s=2n$, Theorem \ref{Thm: invariants under SL} shows that the vector space of invariants in $W^{\otimes s'}\otimes(W^*)^{\otimes (2n-s')}$ is generated by $\det W$ if $s'=2n$, by $\det W^*$ if $s'=0$, by complete contractions of degree $2n$ if $s'=n$, and it is zero in all other cases. 
In the case $s'=2n$, we have
\[
\textstyle Z_{I,2n}=\bigwedge^{i_1}W_1\otimes\cdots\otimes \bigwedge^{i_k}W_k\subseteq W^{\otimes 2n}.
\]
Since $\det W$ is $I$-invariant, we conclude that $\det W\in Z_{I,2n}$ and that the vector space of invariants in $Z_{I,2n}$ is generated by $\det W$. Similarly, if $s'=0$, one sees that $\det W^*$ generate the ring of invariants in $Z_{I,2n}$. Finally, if $s'=n$, one sees that the ring of invariants in $Z_{I,2n}$ is
generated by $I$-alternating linear combinations of complete contractions, as in the previous case.

\smallskip

For $s>2n$, the invariants in $Z_{I,s}$ can be expressed in terms of invariants of $Z_{\tilde{I},\tilde{s}}$ for $\tilde{s}\leq 2n$. This follows from the fact that invariants in  $\left(\bigotimes^{\bullet}( W\oplus W^*)\right)^{\mathrm{SL}(W,\C)}$ can be expressed in terms of invariants of degree $\leq 2n$.

\smallskip

To sum up, we proved that invariants in $\textstyle\left(\bigwedge^{\bullet}((W\oplus W^*)^{\oplus k})\right)^{\mathrm{SL}(W,\C)}$
can be expressed in terms of linear combinations of complete contractions and the images the maps $\bigwedge^{2n} W\longhookrightarrow Z_{I,2n}$ for all $I=(i_1,\ldots,i_{2k})$ such that $i_1+\cdots+i_k=2n$ and $\bigwedge^{2n} W^*\longhookrightarrow Z_{I,2n}$ for all $I=(i_1,\ldots,i_{2k})$ such that $i_1+\cdots+i_k=0$.

\smallskip

Similarly to Remark \ref{rmk.: complete contractions and degree-two complete contractions}, any linear combination of complete contractions can be written as a linear combination of wedge products of degree-two complete contractions in $\bigwedge^2\left( (W\oplus W^*)^{\oplus k}\right)$.
Therefore, to conclude the proof, it suffices to note that the set of all images of the maps $\bigwedge^{2n} W\longhookrightarrow Z_{I,2n}$ for all $I=(i_1,\ldots,i_{2k})$ such that $i_1+\cdots+i_k=2n$ is the set of all realizations of $\det W$ defined
in Remark \ref{rmk.: all possible realizations of determinants}, and similarly for $\det W^*$.
\end{proof}
\end{cor}

Before applying Corollary \ref{Cor: invariants under SL extended version} to study Hodge classes for powers of abelian varieties of Weil type, let us define the set of realizations of exceptional classes:
\begin{rmk}
\label{rmk.: all possible realization of the exceptional classes}
Let $A$ be an abelian variety of $K$-Weil type of dimension $2n$. As in Remark \ref{Rmk.: exceptional classes on A}, we identify $\bigwedge_{K}^{2n} H^1(A,\Q)$ with the set of exceptional classes on $A$ via the natural embedding
\[
\epsilon\colon\textstyle \bigwedge_{K}^{2n} H^1(A,\Q)\longhookrightarrow \bigwedge^{2n} H^1(A,\Q) \simeq H^{2n}(A,\Q).
\]
For any integer $k>1$, similarly to Remark \ref{rmk.: all possible realizations of determinants}, consider various embeddings of $\bigwedge^{2n} H^1(A,\Q)$ into $\bigwedge^{2n} (H^1(A,\Q)^{\oplus k})\simeq H^{2n}(A^k,\Q)$:
If $I=(i_1,\ldots,i_k)$ is a $k$-partition of $2n$, denote by
\[
\textstyle\iota_I\colon \bigwedge^{2n}H^1(A,\Q)\lhook\joinrel\longrightarrow \bigwedge^{i_1}H^1(A,\Q)\otimes\cdots\otimes\bigwedge^{i_k}H^1(A,\Q)\subseteq H^{2n}(A^k,\Q),
\]
the corresponding embedding as in Remark \ref{rmk.: all possible realizations of determinants}. As $\iota_I$ is a morphisms of Hodge structures, $\iota_I(\alpha)$ is a Hodge class for every $\alpha\in \bigwedge_{K}^{2n} H^1(A,\Q)$. We call $\iota_I(\alpha)$ \textit{a realization} of $\alpha$ on $A^k$. Finally, the set of \textit{all realizations} on $A^k$ of the exceptional classes is the set
\[
\{\iota_I(\alpha)\;|\; \alpha\in\textstyle \bigwedge_{K}^{2n} H^1(A,\Q), \text{and } I\; \text{is a } k \text{-partition of } 2n\}\subseteq H^{2n}(A^k,\Q).
\]
\end{rmk}

The following result describes the set of Hodge classes on the powers of $A$ in terms of realizations of exceptional classes of $A$.

\begin{lem}
\label{Lem.: Hodge classes powers of abelian var of weil type}
Let $A$ be a general abelian variety of $K$-Weil type of dimension $2n$. Then, for every positive integer $k$, any Hodge class on $A^k$ can be expressed in terms of rational $(1,1)$-classes and all realizations of the exceptional classes on $A^k$.
\begin{proof}
Let $k$ be a positive integer. As $A$ is a general abelian variety of Weil type, the algebra of Hodge classes on $A^k$ satisfies the following:
\begin{equation}
    \label{eq. Hodge classes A^k}
    \textstyle\left(\bigwedge^{\bullet}(H^1(A,\Q)^{\oplus k})\right)^{\mathrm{Hdg}(A)}\otimes_\Q\C\simeq \textstyle\left(\bigwedge^{\bullet}((W\oplus W^*)^{\oplus k})\right)^{\mathrm{SL}(W,\C)}.
\end{equation}
By Corollary \ref{Cor: invariants under SL extended version}, any invariant in this algebra can be expressed in terms of invariants of degree two and realizations of $\det W$ and $\det W^*$. The $\C$-linear span of invariants of degree two is equal to the $\C$-linear span of rational $(1,1)$-classes on $A^k$, hence, it suffices to show that the $\C$-linear span of the set of all realizations of $\det W$ and $\det W^*$ in $\bigwedge^{\bullet}((W\oplus W^*)^{\oplus k})$ is equal to the $\C$-linear span of the set of all realizations of the exceptional classes on $A^k$ defined in Remark $\ref{rmk.: all possible realization of the exceptional classes}$. This follows from the fact that, if $I$ is any $k$-partition of $2n$, the isomorphism in (\ref{eq. Hodge classes A^k}) sends the $\C$-linear span of $\iota_I\left(\bigwedge^{2n}_K H^1(A,\Q)\right)\subseteq H^{2n}(A^k,\Q)$ onto the
$\C$-linear span of \[\textstyle\iota_I\left(\langle \bigwedge^{2n} W,\bigwedge^{2n}W^*\rangle\right)\subseteq \bigwedge^{2n}\left((W\otimes W^*)^{\oplus k}\right).\qedhere\]
\end{proof}
\end{lem}
In particular, to show that the Hodge conjecture for $A$ implies the Hodge conjecture for all powers of $A$, it suffices to show that if an exceptional class is algebraic on $A$ then all its realizations are algebraic. 
The following lemma gives some relations between the realizations of the exceptional classes:
\begin{lem}
\label{Lem.: realizations of exceptional classes are related}
Let $A$ be a general abelian variety of $K$-Weil type of dimension $2n$, and let $\alpha$ be an exceptional class on $A$. Denote by $\widetilde{\alpha}$ the realization of $\alpha$ on $A^{2n}$ via the embedding
\[
\textstyle
\bigwedge^{2n}_{K}H^1(A,\Q)\longhookrightarrow H^1(A,\Q)\otimes\cdots\otimes H^1(A,\Q)\subseteq H^{2n}(A^{2n},\Q).
\]
Then, for any positive integer $k$, and any $k$-partition $I$ of $2n$, the realization $\alpha_I$ on $A^k$ is the pullback of $\widetilde{\alpha}$, via an algebraic map $A^k\longrightarrow A^{2n}$. 
In particular, if $\widetilde{\alpha}$ is algebraic on $A^{2n}$, then any realization of $\alpha$ on any power of $A$ is algebraic.
\begin{proof}
Writing $\alpha$ as a sum of decomposable elements of $\bigwedge^{2n}H^1(A,\Q)$, we see that it suffices to show that, for any decomposable element $\beta\coloneqq v_1\wedge\ldots \wedge v_{2n}\in \bigwedge^{2n}H^1(A,\Q)$, any realization of $\beta$ on $A^k$ is the pullback via an algebraic map $A^k\longrightarrow A^{2n}$ of 
\[\textstyle\widetilde{\beta}=\sum_\s \pm v_{\s(1)}\otimes\cdots\otimes v_{\s(2n)},\]
where the sum runs over all permutations $\s$ of $\{1,\ldots,2n\}.$ 
Let $I=(i_1,\ldots,i_k)$ be a $k$-partition of $2n$, and let $\beta_I$ be the realization of $\beta$ on $A^k$ via the embedding
\[
\textstyle
\bigwedge^{2n}_{K}H^1(A,\Q)\longhookrightarrow \bigwedge^{i_1}H^1(A,\Q)\otimes \cdots\otimes \bigwedge^{i_k}H^1(A,\Q)\subseteq H^{2n}(A^k,\Q).
\]
Then, one sees that a multiple of $\beta_I$ is equal to the pullback of $\widetilde{\beta}$ via the map 
\[\Delta_{i_1}\times \cdots\times\Delta_{i_k}\colon A^{k}\longrightarrow A^{2n},\]
where $\Delta_{i_j}$ is the diagonal map $A\longrightarrow A^{i_j}$.
\end{proof}
\end{lem}
We have now everything needed to prove that the Hodge conjecture for a general abelian variety of Weil type implies the Hodge conjecture for all its powers.
\begin{thm}\cite{abdulali1999abelianIII}
\label{Thm.: H.c. for powers of general ab var of weil type}
Let $A$ be a general abelian variety of $K$-Weil type. Then, the Hodge conjecture for $A$ implies the Hodge conjecture for all powers $A^k$.
\begin{proof}
By Lemma \ref{Lem.: Hodge classes powers of abelian var of weil type}, for any positive integer $k$, any Hodge class on $A^k$ can be expressed in terms of rational $(1,1)$-classes and realizations of the exceptional classes on $A^k$. Since rational $(1,1)$-classes are algebraic by the Lefschetz $(1,1)$ theorem, to prove the Hodge conjecture for all powers of $A$, we need to show that all realizations of the exceptional classes on the powers of $A$ are algebraic. By Lemma \ref{Lem.: realizations of exceptional classes are related}, it suffices to show that for any exceptional class $\alpha$ on $A$, its realization $\widetilde{\alpha}$ via the map 
\[
\textstyle\bigwedge_K^{2n} H^1(A,\Q)\longhookrightarrow H^1(A,\Q)^{\otimes 2n}\subseteq H^{2n}(A^{2n},\Q)
\]
is algebraic.

\smallskip

To do this, consider the following maps: For $J\subseteq\{1,\ldots,2n\}$, let $p_J\colon A^{2n}\longrightarrow A^{|J|}$ be the projection from $A^{2n}$ onto the $J$-th components, and, for $i\geq 1$, let $\Sigma_i\colon A^i\longrightarrow A$ be the summation map. In cohomology, the pullback via $\Sigma_i$ is
\[
\Sigma_i^*\colon H^1(A,\Q)\longrightarrow H^1(A,\Q)^{\oplus i}, \quad v\longmapsto (v,\ldots,v).
\]
With this notation, we will prove that, for every $\beta\in \bigwedge^{2n}H^1(A,\Q)$, the following equality holds:
\begin{equation}
\label{Eq.: relating different determinants}
   \Sigma_{2n}^*(\beta)=\sum_{\emptyset\not =J\subsetneq\{1,\ldots,2n\}} (-1)^{|J|-1}p_J^*(\Sigma_{|J|}^*(\beta))+\widetilde{\beta}\subseteq\textstyle H^{2n}(A^{2n},\Q),
\end{equation}
where $\widetilde{\beta}$ is the image of $\beta$ via the natural map $\bigwedge^{2n}H^1(A,\Q)\longhookrightarrow H^1(A,\Q)^{\otimes 2n}\subset H^{2n}(A^{2n},\Q)$.
For $\beta=\alpha$, we will then conclude from (\ref{Eq.: relating different determinants}) that $\widetilde{\alpha}$ is algebraic since $\alpha$ is assumed to be algebraic and all the maps involved are algebraic.

\smallskip
Note that by linearity, it suffices to prove
(\ref{Eq.: relating different determinants}) in the case where $\beta$ is decomposable, i.e., $\beta= v_1\wedge\cdots\wedge v_{2n}\in \bigwedge^{2n}H^1(A,\Q)$. 
By the commutativity of pullbacks and cup products, we have the following equalities:
\[
\Sigma_{2n}^*(\beta)=\Sigma^*_{2n}(v_1)\wedge\cdots\wedge \Sigma^*_{2n}(v_{2n})=(v_1,\ldots,v_1)\wedge\cdots\wedge (v_{2n},\ldots,v_{2n})=\textstyle \sum (p_{i_1}^*(v_1)\wedge\ldots\wedge p_{i_{2n}}^*(v_{2n})),
\]
where the sum runs over all $i_1,\ldots,i_{2n}\in \{1,\ldots,2n\}$.
To prove (\ref{Eq.: relating different determinants}), we show that each of these summands appears exactly once on the right-hand-side of (\ref{Eq.: relating different determinants}) after simplifying it. Let us start with the summands for which all $i_j$ are different. These terms do not appear in the sum on the right-hand side of the equality, indeed they do not come from the pullback under the projection onto some lower power of $A$, but they appear exactly once in $\widetilde{\beta}$ since by construction
\[
\textstyle
\widetilde{\beta}=\sum_\s \pm v_{\sigma(1)}\otimes\cdots\otimes v_{\sigma(2n)}=\sum_\s \pm p_1^*( v_{\sigma(1)})\wedge\cdots\wedge p_{2n}^*(v_{\sigma(2n)}).
\]
If $i_j$ are not all different, let $I\coloneqq\{i_1,\ldots,i_{2n}\}$. The term $ p_{i_1}^*(v_1)\wedge\ldots\wedge p_{i_{2n}}^*(v_{2n})$ appears on the right-hand side of (\ref{Eq.: relating different determinants}) once for every $J$ such that $I\subseteq J$. Therefore, taking into account the sign $(-1)^{|J|-1}$ the term $ p_{i_1}^*(v_1)\wedge\ldots\wedge p_{i_{2n}}^*(v_{2n})$ appears exactly once, since
\[
\sum_{i=|I|}^{2n-1} (-1)^{i-1}\binom{2n-|I|}{i-|I|}=1.
\]
This concludes the proof of (\ref{Eq.: relating different determinants}).
\end{proof}
\end{thm}

\begin{rmk}
\label{rmk.: comparison with Abdulali}
As mentioned in the introduction, Theorem \ref{Thm.: H.c. for powers of general ab var of weil type} can already be found in \cite{abdulali1999abelianIII}.
However, in the proof, the author applied an incomplete version of Corollary \ref{Cor: invariants under SL extended version} which does not mention the different realizations of the two determinants $\det W$ and $\det W^*$. As we have seen, the different realizations of the two determinants are linked to the existence of the different realizations of the exceptional classes on the powers of the abelian variety. A priori, the fact that if one realization of an exceptional class is algebraic then all the realizations of it are algebraic was not clear to us.
For a motivic proof of Theorem \ref{Lem.: Hodge classes powers of abelian var of weil type} see \cite{milne2021tate}.
\end{rmk}

A result similar to Theorem \ref{Thm.: H.c. for powers of general ab var of weil type} holds also for abelian varieties of Weil type with definite quaternionic multiplication. Note that one needs to add to the proof the same modifications as we introduced in the proof of Theorem \ref{Thm.: H.c. for powers of general ab var of weil type}.

\begin{thm}\cite[Thm.\ 4.1]{abdulali1999abelianIII},\cite[Thm.\ 4.2.1]{schlickewei2010hodge}
\label{Thm.: Hodge conjecture abelian variety of Weil type with definite quaternionic multiplication}
Let $A$ be a general abelian variety of $K$-Weil type with definite quaternionic multiplication. Then, if the Weil classes on $A$ are algebraic, the Hodge conjecture holds for every power $A^k$.
\end{thm}

Theorem \ref{Thm.: H.c. for powers of general ab var of weil type} and Theorem \ref{Thm.: Hodge conjecture abelian variety of Weil type with definite quaternionic multiplication} allow us to deduce the Hodge conjecture for all powers $A^k$, where $A$ is a general abelian variety of Weil type or a general abelian variety of Weil type with definite quaternionic multiplication whenever the Weil classes on $A$ are algebraic.

\section{Families of K3 surfaces of general Picard number 16}
\label{Sec. Families of K3 of Pic 16}
In this section, we prove the Hodge conjecture for all powers of the K3 surfaces belonging to families whose general element has Picard number $16$ assuming the Kuga--Satake Hodge conjecture.

\smallskip

The relation between K3 surfaces of Picard number $16$ and abelian fourfolds of Weil type is given by the Kuga--Satake construction due to the following result:

\begin{thm}\cite{lombardo01kugasatake},\cite[Thm.\ 9.2]{van2000kuga}
\label{thm: Lombardo KS and rank 6}
Let $(V,q)$ be a polarized rational Hodge structure of K3-type such that there exists an isomorphism of quadratic spaces 
\[
(V,q)\simeq U_\Q^2 \oplus \langle a \rangle \oplus \langle b \rangle,
\]
where $a$ and $b$ are negative integers, and $U$ is the hyperbolic lattice. Then, the Kuga--Satake variety of $(V,q)$ is isogenous to $A^4$ for an abelian fourfold of $\Q(\sqrt{-ab})$-Weil type with discriminant one. Moreover, for general $(V,q)$, there is an isomorphism $\mathrm{End}(A)\otimes_\mathbb{Z}\Q\simeq \Q(\sqrt{-ab})$.
Conversely, if $A$ is an abelian fourfold of Weil type with discriminant one, then $A^4$ is the Kuga--Satake variety of a polarized rational Hodge structure of K3-type of dimension six as above.
\end{thm}
By Theorem \ref{Thm.: Markman}, the Hodge conjecture holds for the general abelian fourfold of Weil type with discriminant one. The following proposition shows that the same result follows in the cases where Kuga--Satake correspondence is known to be algebraic. As mentioned in the introduction, this does not prove the Hodge conjecture for any new abelian variety, it just shows the strong relation between the algebraicity Kuga--Satake correspondence and the Hodge conjecture for the Kuga--Satake variety in this special case.
\begin{prop}
\label{prop. KSH implies H conj for ab 4}
Let $A$ be a general abelian fourfold of $K$-Weil type with discriminant one. If the Kuga--Satake Hodge conjecture holds for the corresponding K3 surface, then the Weil classes on $A$ are algebraic. Thus, the Hodge conjecture holds for $A$ and, hence, for all powers $A^k$.
\begin{proof}
Let $X$ be a K3 surface such that $\mathrm{KS}(X)\sim A^4$ which exists by Theorem \ref{thm: Lombardo KS and rank 6}. By \cite[Thm.\ 3.8]{lombardo01kugasatake}, we have the following isomorphism of Hodge structures:
\begin{equation}
\label{eq.: T + F(T)}
\textstyle
\bigwedge^2_K H^1(A,\Q)\simeq T(X)\oplus \phi(T(X)),
\end{equation}
where $\phi$ denotes the natural action of $\sqrt{-d}$ on $\bigwedge^2_K H^1(A,\Q)$ sending $v\wedge_K w\longmapsto \sqrt{-d}v\wedge_K w$. 
In particular, there is an embedding of Hodge structures
\[
\textstyle\kappa\colon T(X)\longhookrightarrow\bigwedge^2_K H^1(A,\Q)\longhookrightarrow \bigwedge^2 H^1(A,\Q).
\]
By generality assumption, the K3 surface $X$ is Mumford--Tate general. In particular, \cite[Thm.~4.8]{charles2022two} and its proof show that $\kappa$ is induced by the Kuga--Satake correspondence. As we are assuming the Kuga--Satake Hodge conjecture for $X$ we deduce that $\kappa$ is algebraic.
Tensoring (\ref{eq.: T + F(T)}) by $\C$, we get as in \cite[Prop.\ 4.4]{deligne1982hodge} the following isomorphism
\begin{equation}
\label{eq: two ways of computing the same thing}
\textstyle
\bigwedge^2 W\oplus \bigwedge^2 W^*\simeq\left(\bigwedge^2_K H^1(A,\Q)\right)\otimes_\Q \C\simeq T(X)_\C\oplus \phi(T(X))_\C,
\end{equation}
where, as usual, $W$ is a complex vector space of dimension $2n=4$ such that $\mathrm{Hdg}(A)(\C)\simeq\mathrm{SL}(W,\C).$
Using the left-hand side of (\ref{eq: two ways of computing the same thing}), we compute the ring of Hodge classes in
\[
\textstyle S\coloneqq\left(\bigwedge^2_K H^1(A,\Q)\right)_\C\otimes \left(\bigwedge^2_K H^1(A,\Q)\right)_\C.
\]
We see that it is four-dimensional and is spanned by two linear combinations of complete contractions together with two realizations on $A^2$ of the Weil classes. On the other hand, computing the same ring of Hodge classes using the right-hand side of (\ref{eq: two ways of computing the same thing}), we conclude for dimension reasons that there is a unique Hodge class in $T(X)\otimes T(X)$, i.e., that the endomorphism field of $X$ is $\Q$. Denoting by $\alpha$ this unique Hodge class, we see that the ring of Hodge classes in $S$ is spanned by 
\[\alpha, (\mathrm{Id}\otimes \phi) \alpha , (\phi\otimes\mathrm{Id})\alpha, \text{ and } (\phi\otimes \phi)\alpha.\] 
Since we are assuming the Kuga--Satake Hodge conjecture, the class $\alpha$ is algebraic on $A^2$ as it is represented on $X^2$ by a component of the class of the diagonal.
Moreover, note that $\phi$ is algebraic as it is the restriction to $\bigwedge^2_K H^1(A,\Q)$ of the algebraic morphism \[\sqrt{-d}\otimes \mathrm{Id}\colon H^1(A,\Q)^{\otimes 2}\longrightarrow H^1(A,\Q)^{\otimes 2}.\]
We then conclude that every class in $S$ is algebraic. In particular, we see that the realizations of the Weil classes in $S$ are algebraic. Using a similar argument to the one of Lemma \ref{Lem.: realizations of exceptional classes are related}, we conclude that the Weil classes on $A$ are algebraic. This implies by Theorem \ref{Thm.: H.c. for powers of general ab var of weil type} the Hodge conjecture for all powers $A^k$.
\end{proof}
\end{prop}

We are finally able to state and prove our main theorem which extends \cite[Thm.\ 2]{schlickewei2010hodge}.
\begin{thm}
\label{main thm}
Let $\mathscr{X}\longrightarrow S$ be a four-dimensional family of K3 surfaces whose general fibre is of Picard number $16$ with an isometry
\[
T(\mathscr{X}_s)\simeq U_\Q^2\oplus \langle a\rangle\oplus\langle b\rangle,
\]
for some negative integers $a$ and $b$. If the Kuga--Satake correspondence is algebraic for the fibres of this family, then the Hodge conjecture holds for all powers of every K3 surface in this family.
\begin{proof}
Recall that, if a K3 surface $X$ has totally real endomorphism field $E$, the dimension of $T(X)$ as an $E$-vector space is at least three. Using this observation, together with the assumption that the general K3 surface of this family has Picard number $16$, we see that,
for any $s\in S$ the pair $(\dim T(\X_s),E\coloneqq \mathrm{End}_{\mathrm{Hdg}(\X_s)}(T(\X_s))$ satisfies one of the following:
\begin{itemize}
    \item[(i)] $\dim T(\X_s)=6, 4,$ or $2$ and $E$ is a CM field;
    \item[(ii)] $\dim T(\X_s)=6$ and $E=\Q$;
    \item[(iii)] $\dim T(\X_s)=6$ and $E$ is a totally real field of degree two;
    \item[(iv)] $\dim T(\X_s)\leq 5$ and $E\simeq \Q$.
\end{itemize}

In case (i), the endomorphism field of $T(\X_s)$ is a CM field. Therefore, we may apply Corollary \ref{Cor.: H.c. for K3 of CM type} to deduce the Hodge conjecture for all powers of $\X_s$.

\smallskip

In case (ii), the transcendental lattice is six-dimensional and by Theorem \ref{thm: Lombardo KS and rank 6}, the Kuga--Satake variety of $\X_s$ is the fourth power of a general abelian fourfold $A$ of Weil type. By Theorem \ref{Thm.: Markman} and Theorem \ref{Thm.: H.c. for powers of general ab var of weil type}, the Weil classes on $A$ are algebraic and the Hodge conjecture holds for all powers of $A$. Using Lemma \ref{Lem.: Kuga-Satake corresponedence and Hodge conjecture}, we conclude that all Hodge classes in the tensor algebra of $T(\X_s)$ are algebraic on the powers of $\X_s$.

\smallskip

In case (iii), one knows that the abelian fourfold $A$ appearing in the decomposition of the Kuga--Satake variety of $\X_s$ is an abelian fourfold of Weil type whose endomorphism algebra is of definite quaternion type, see \cite[Prop.\ 5.7]{van2000kuga}. By Theorem \ref{Thm.: H.c. for powers of general ab var of weil type} together with Theorem \ref{Thm.: Markman}, we see as in the previous case that the Hodge conjecture holds for all powers of the general $A$. As before, this is sufficient to conclude that all Hodge classes on the tensor algebra of $T(\X_s)$ are algebraic.

\smallskip

Finally in case (iv), the endomorphism field of $\X_s$ is $\Q$. Since we have already proven the Hodge conjecture for all powers of the K3 surfaces of Picard number $16$ belonging to this family, we may apply our degeneration result of Corollary \ref{Cor.: when the determinant is the only class to study} to conclude that the Hodge conjecture holds for all powers of the K3 surface $\X_s$.

\smallskip

This concludes the proof.
\end{proof}
\end{thm}

We end this paper recalling two four-dimensional families of K3 surfaces satisfying the hypothesis of Theorem \ref{main thm}.

\begin{exmp}[\textit{Double covers of $\mathbb{P}^2$ branched along six lines}]
This family of K3 surfaces has been first studied by Paranjape \cite{paranjape} and it is the example studied by Schlickewei \cite{schlickewei2010hodge}: 
Let $\pi\colon Y \longrightarrow \p^2$ be a double cover branched along six lines no three of which intersect in one point. The surface $Y$ has simple nodes in the $15$ points of intersection of the lines. Blowing up these $15$ points on $Y$ we get a smooth K3 surface $X$. The $15$ exceptional lines on $X$ together with pullback of the ample line on $\p^2$ span a sublattice of $\mathrm{NS}(X)$ of rank $16$. Since this family of K3 surfaces is four-dimensional, the Picard number of a general member is $16$. The transcendental lattice of the general element of this family has been computed in \cite[Lem.\ 1]{paranjape}, where it is shown that it is isomorphic to 
$U_\Q^2\oplus\langle-2\rangle^2$.
In particular, by Theorem \ref{thm: Lombardo KS and rank 6} the Kuga--Satake variety is the fourth power of an abelian fourfold of $\Q(i)$-Weil type.
This was already known to Paranjape \cite{paranjape} where the author constructs the Kuga--Satake correspondence for this family, covering the general K3 surface by the square of a curve of genus five. Our Theorem \ref{main thm} then extends the result in \cite{schlickewei2010hodge}, since it allows us to conclude that the Hodge conjecture holds for all powers of the K3 surfaces in this family and not just for their square.
\end{exmp}

\begin{exmp}[\textit{Desingularization of K3 surfaces in $\p^4$ with $15$ simple nodes}]
This family of K3 surfaces has been first introduced in \cite{garbagnati2016kummer}. In \cite{ingalls2022explicit}, the authors show that the same techniques as in \cite{schlickewei2010hodge} can be used to prove the Hodge conjecture for the square of these K3 surfaces:
Let $X$ be a general K3 surface which is the desingularization of a singular K3 surface in $\p^4$ with $15$ nodal points. Then, the $15$ rational lines on $X$ together with the pullback of the ample line bundle on $\p^4$ span a sublattice of $\mathrm{NS}(X)$ of rank $16$. Therefore, since the family of such K3 surfaces is four-dimensional, the Picard number of $X$ is equal to $16$.
In \cite[Rmk.\ 4.8]{ingalls2022explicit}, using elliptic fibrations, it is shown that the transcendental lattice of a general K3 surface is isomorphic to
$U^{\oplus 2}_\Q\oplus\langle -6\rangle\oplus\langle -2\rangle.$
In particular, applying Theorem \ref{thm: Lombardo KS and rank 6}, the Kuga--Satake variety of $X$ is isogenous to $A^4$ where $A$ is an abelian fourfold of $\Q(\sqrt{-3})$-Weil type. Inspired by \cite{paranjape}, Ingalls, Logan, and Patashnick \cite{ingalls2022explicit} show that the Kuga--Satake correspondence is algebraic for these K3 surfaces. The authors then show that the proof of Schlickewei applies which shows the Hodge conjecture for the square of these K3 surfaces. As in the previous example, by Theorem \ref{main thm}, we conclude that the Hodge conjecture holds for all powers of the K3 surfaces in this family.
\end{exmp}

\bibliography{main}

\begin{thebibliography}{10}

\bibitem{abdulali1999abelianIII}
Salman Abdulali.
\newblock Abelian varieties of type {I}{I}{I} and the {H}odge conjecture.
\newblock {\em Int.\ J.\ Math.}, 10(06):667--675, 1999.

\bibitem{buskin2019every}
Nikolay Buskin.
\newblock Every rational {H}odge isometry between two {K}3 surfaces is
  algebraic.
\newblock {\em J.\ reine angew.\ Math.}, 2019(755):127--150, 2019.

\bibitem{charles2022two}
Fran{\c{c}}ois Charles.
\newblock Two results on the {H}odge structure of complex tori.
\newblock {\em Math.\ Z.}, 300(4):3623--3643, 2022.

\bibitem{deligne1982hodge}
Pierre Deligne.
\newblock Hodge cycles on abelian varieties.
\newblock In {\em Hodge cycles, motives, and Shimura varieties}, pages 9--100.
  Springer, 1982.

\bibitem{garbagnati2016kummer}
Alice Garbagnati and Alessandra Sarti.
\newblock Kummer surfaces and {K}3 surfaces with $(\mathbb{Z}/2\mathbb{Z})^4$
  symplectic action.
\newblock {\em Rocky Mt.\ J.\ Math.}, 46(4):1141--1205, 2016.

\bibitem{huybrechtsK3surfaces}
Daniel Huybrechts.
\newblock {\em Lectures on {K}3 surfaces}, volume 158.
\newblock Cambridge Univ.\ Press, 2016.

\bibitem{huybrechts2019motives}
Daniel Huybrechts.
\newblock Motives of isogenous {K}3 surfaces.
\newblock {\em Comment.\ Math.\ Helvetici}, 94(3):445--458, 2019.

\bibitem{ingalls2022explicit}
Colin Ingalls, Adam Logan, and Owen Patashnick.
\newblock Explicit coverings of families of elliptic surfaces by squares of
  curves.
\newblock {\em Math.\ Z.}, 302(2):1191--1238, 2022.

\bibitem{jannsen1992motives}
Uwe Jannsen.
\newblock Motives, numerical equivalence, and semi-simplicity.
\newblock {\em Invent.\ Math.}, 107(3):447--452, 1992.

\bibitem{kleiman1968algebraic}
Steven~L. Kleiman.
\newblock {\em Algebraic cycles and the Weil conjectures}.
\newblock Columbia Univ., Department of mathematics, 1968.

\bibitem{kraft1996classical}
Hanspeter Kraft and Claudio Procesi.
\newblock Classical invariant theory, a primer.
\newblock {\em Lecture Notes. Preliminary version}, 1996.

\bibitem{lombardo01kugasatake}
Giuseppe Lombardo.
\newblock Abelian varieties of {W}eil type and {K}uga--{S}atake varieties.
\newblock {\em Tohoku Math.\ J., Second Series}, 53(3):453--466, 2001.

\bibitem{markman2022monodromy}
Eyal Markman.
\newblock The monodromy of generalized {K}ummer varieties and algebraic cycles
  on their intermediate {J}acobians.
\newblock {\em J.\ Eur.\ Math.\ Soc.}, 2022.

\bibitem{milne2021tate}
James~S. Milne.
\newblock The {T}ate and {S}tandard {C}onjectures for {C}ertain {A}belian
  {V}arieties.
\newblock {\em arXiv preprint arXiv:2112.12815}, 2021.

\bibitem{moonen1995hodge}
Ben Moonen and Yuri~G. Zarhin.
\newblock Hodge classes and {T}ate classes on simple abelian fourfolds.
\newblock {\em Duke Math.\ J.}, 77(3):553--581, 1995.

\bibitem{moonen1999hodge}
Ben Moonen and Yuri~G. Zarhin.
\newblock Hodge classes on abelian varieties of low dimension.
\newblock {\em Math.\ Ann.}, 4(315):711--733, 1999.

\bibitem{morrison1984k3}
David~R. Morrison.
\newblock On {K}3 surfaces with large {P}icard number.
\newblock {\em Invent.\ Math}, 75:105--121, 1984.

\bibitem{paranjape}
Kapil Paranjape.
\newblock Abelian varieties associated to certain {K}3 surfaces.
\newblock {\em Compos.\ Math.}, 68(1):11--22, 1988.

\bibitem{mari2008hodge}
Jos{\'e} Ram{\'o}n-Mar{\'\i}.
\newblock On the {H}odge conjecture for products of certain surfaces.
\newblock {\em Collect.\ Math.}, 59(1):1--26, 2008.

\bibitem{ribet1983hodge}
Kenneth~A. Ribet.
\newblock Hodge classes on certain types of abelian varieties.
\newblock {\em Am.\ J.\ Math.}, 105(2):523--538, 1983.

\bibitem{schlickewei2010hodge}
Ulrich Schlickewei.
\newblock The {H}odge conjecture for self-products of certain {K}3 surfaces.
\newblock {\em J.\ Algebra}, 324(3):507--529, 2010.

\bibitem{schoen1988hodge}
Chad Schoen.
\newblock Hodge classes on self-products of a variety with an automorphism.
\newblock {\em Compos.\ Math.}, 65(1):3--32, 1988.

\bibitem{van1994introduction}
Bert van Geemen.
\newblock An introduction to the {H}odge conjecture for abelian varieties.
\newblock In {\em Algebraic Cycles and Hodge Theory}, pages 233--252. Springer,
  1994.

\bibitem{van1996theta}
Bert van Geemen.
\newblock Theta functions and cycles on some abelian fourfolds.
\newblock {\em Math.\ Z}, 221(4):617--631, 1996.

\bibitem{van2000kuga}
Bert van Geemen.
\newblock Kuga--{S}atake varieties and the {H}odge conjecture.
\newblock In {\em The arithmetic and geometry of algebraic cycles}, pages
  51--82. Springer, 2000.

\bibitem{weil}
Andr\'{e} Weil.
\newblock Abelian varieties and the {H}odge ring.
\newblock In {\em Collected Papers}, volume III, pages 9--100. Springer, 1982.

\bibitem{h.weyl2016classical}
Hermann Weyl.
\newblock {\em The classical groups}.
\newblock Princeton Univ.\ Press, 2016.

\bibitem{zarhin1983hodge}
Yuri~G. Zarhin.
\newblock Hodge group of {K}3 surfaces.
\newblock {\em J.\ reine angew.\ Math}, 391:193--220, 1983.

\end{thebibliography}
\bibliographystyle{plain}
\end{document}